\documentclass[11pt,reqno]{amsart}%
\usepackage{amsmath, amsfonts, amssymb, amsthm, amscd, amsbsy}
\usepackage{fancyhdr}
\usepackage[usenames,dvipsnames,svgnames,x11names,hyperref]{xcolor}
\usepackage{geometry}
\usepackage{graphicx}
\usepackage{hyperref}
\usepackage{amsmath}
\usepackage{amsfonts}
\usepackage{amssymb}%
\providecommand{\U}[1]{\protect\rule{.1in}{.1in}}
\hypersetup{backref=true,
pagebackref=true,
hyperindex=true,
colorlinks=true,
breaklinks=true,
urlcolor=NavyBlue,
linkcolor=Fuchsia,
bookmarks=true,
bookmarksopen=false,
filecolor=black,
citecolor=ForestGreen,
linkbordercolor=red
}
\newtheorem{theorem}{Theorem}[section]
\theoremstyle{plain}

\newtheorem{corollary}{Corollary}[section]

\newtheorem{definition}{Definition}[section]

\newtheorem{lemma}{Lemma}[section]

\newtheorem{remark}{Remark}[section]

\numberwithin{equation}{section}
\allowdisplaybreaks
\newtheorem{theorema}{Theorem}[section]

\AtBeginDocument{\hypersetup{pdfborder={0 0 0.01}}}
\DeclareMathOperator{\Vol}{Vol}
\begin{document}
\title[logarithmic Sobolev inequalities]{Logarithmic Sobolev, Poincar\'e and Beckner Inequalities on Hyperbolic Spaces and Riemannian Manifolds}
\author{Anh Xuan Do}
\address{Anh Xuan Do: Department of Mathematics, University of Connecticut, Storrs, CT
06269, USA}
\email{anh.do@uconn.edu}
\author{Debdip Ganguly}
\address{Debdip Ganguly: Theoretical Statistics and Mathematics Unit\\
Indian Statistical Institute, Delhi Centre\\
S.J. Sansanwal Marg, New Delhi, Delhi 110016, India}
\email{debdip@isid.ac.in}
\author{Nguyen Lam}
\address{Nguyen Lam: School of Science and the Environment, Grenfell Campus, Memorial
University of Newfoundland, Corner Brook, NL A2H5G4, Canada}
\email{nlam@mun.ca}
\author{Guozhen Lu}
\address{Guozhen Lu: Department of Mathematics, University of Connecticut, Storrs, CT
06269, USA}
\email{guozhen.lu@uconn.edu}

\begin{abstract}
We investigate several functional and geometric inequalities on the hyperbolic space $\mathbb{H}^N$, with a primary emphasis on logarithmic Sobolev inequalities, Poincaré inequalities, and Beckner-type inequalities, all studied within the framework of the $AB$ program. The main analytical tool employed throughout this paper is symmetrization. More precisely, our approach relies on an improved version of the P\'olya–Szeg\"o inequality on the hyperbolic space, obtained through a careful comparison of the gradient norms of rearranged functions in the hyperbolic and Euclidean settings.  

For Beckner-type inequalities, we adopt a semigroup approach based on sharp estimates for the heat semigroup, leading to refined interpolation inequalities between Poincaré and logarithmic Sobolev inequalities. Finally, we extend our results beyond hyperbolic space to a class of Riemannian model manifolds $\mathbb{M}^N$ satisfying the \emph{centered isoperimetric inequality}. This shows that the inequalities and methods developed in this work are robust and rely mainly on geometric and isoperimetric properties, rather than on the specific structure of hyperbolic space itself.

\end{abstract}
\maketitle

\section{Introduction}\label{secintro}
The primary objective of this paper is to investigate a range of  functional and geometric inequalities on hyperbolic spaces. In particular, we focus on logarithmic Sobolev inequalities, Poincar\'e inequalities, and Beckner-type inequalities. We analyze their validity, and sharp constants,  and explore how the underlying hyperbolic geometry influences these inequalities in comparison with the Euclidean setting.

Throughout this paper, unless explicitly stated otherwise, we work with the Poincar\'e ball model of the hyperbolic space \(\mathbb{H}^N\). In this model, \(\mathbb{H}^N\) is identified with the Euclidean unit ball
\[
B^N := \{ x \in \mathbb{R}^N : |x|^2 < 1 \},
\]
endowed with the Riemannian metric
\[
g(x) = \left( \frac{2}{1 - |x|^2} \right)^2 dx \otimes dx,
\]
where \(|x|^2 = \sum_{i=1}^N x_i^2\) denotes the square of the standard Euclidean norm. Equivalently, the metric \(g\) is conformal to the Euclidean metric \(dx^2\) with conformal factor \(\frac{2}{1 - |x|^2}\).

By definition, \(\mathbb{H}^N\) is a complete, non-compact Riemannian manifold of dimension \(N\) with constant sectional curvature equal to \(-1\). It is well known that any two such Riemannian manifolds are isometric; see, for instance, \cite{RAT}. Unless otherwise specified, all computations in this article are carried out within the Poincar\'e ball model, and the notation \(\mathbb{H}^N\) will always refer to this realization.

The Riemannian inner product on the tangent space of \(\mathbb{H}^N\) is denoted by \(\langle \cdot, \cdot \rangle_{\mathbb{H}^N}\), or simply by \(\langle \cdot, \cdot \rangle\) when no ambiguity arises. The associated volume element is
\[
dV_{\mathbb{H}} = \left( \frac{2}{1 - |x|^{2}} \right)^{N} dx,
\]
where $dx$ represents the Lebesgue measure on $\mathbb{R}^{N}$.

Let $\nabla_{\mathbb{H}}$ denote the gradient vector field on $\mathbb{H}^{N}$. In local coordinates, we have
\[
\nabla_{\mathbb{H}} = \left( \frac{1 - |x|^{2}}{2} \right)^{2} \nabla,
\]
where $\nabla$ is the standard Euclidean gradient. For any smooth vector field $X$ on $\mathbb{H}^{N}$, the divergence $\operatorname{div} X$ is a smooth function uniquely defined by the following theorem:

\begin{theorema}[Divergence Theorem]
Let $X$ be a smooth vector field on the hyperbolic space $\mathbb{H}^N$. Then there exists a unique smooth function, denoted by $\operatorname{div} X$, such that for every $u \in C_c^\infty(\mathbb{H}^N)$,
\[
\int_{\mathbb{H}^N} (\operatorname{div} X)\, u \, dV_{\mathbb{H}}
=
- \int_{\mathbb{H}^N} \langle X, \nabla_{\mathbb{H}} u \rangle \, dV_{\mathbb{H}}.
\]
\end{theorema}


In local coordinates, the divergence of a vector field $X$ is given by
\[
\operatorname{div} X
=
\frac{1}{\sqrt{\det g}}
\frac{\partial}{\partial x^{i}}
\left( X^{i} \sqrt{\det g} \right).
\]
See, for instance,~\cite{Gre24, Jost}.
Using this definition, the Laplace--Beltrami operator $\Delta_{\mathbb{H}}$ on $\mathbb{H}^N$ is defined by
\[
\Delta_{\mathbb{H}} := \operatorname{div} \circ \nabla_{\mathbb{H}^N} .
\]
In local coordinates, it admits the explicit representation
\[
\Delta_{\mathbb{H}}
=
\left( \frac{1 - |x|^{2}}{2} \right)^{2} \Delta
+
(N-2)\left( \frac{1 - |x|^{2}}{2} \right) x \cdot \nabla,
\]
where $\nabla$ and $\Delta$ denote the standard Euclidean gradient and Laplacian, respectively, and $x \cdot \nabla$ stands for the Euclidean inner product.

The hyperbolic distance between two points $x,y \in \mathbb{H}^N$ is denoted by $d(x,y)$. 
In particular, the distance between a point $x$ and the origin is given by
\[
\rho(x) := d(x,0)
= \int_{0}^{|x|} \frac{2}{1-s^{2}}\,ds
= \log\!\left(\frac{1+|x|}{1-|x|}\right).
\]
Consequently,
\[
|x| = \tanh\!\left(\frac{\rho(x)}{2}\right),
\]
and
\[
1-|x|^{2}
= 1-\tanh^{2}\!\left(\frac{\rho(x)}{2}\right)
= \frac{4e^{\rho(x)}}{(1+e^{\rho(x)})^{2}}.
\]

For arbitrary points $x,y \in \mathbb{H}^N$, the hyperbolic distance is explicitly given by
\[
d(x,y)
= \cosh^{-1}\!\left(
1 + \frac{2|x-y|^{2}}{(1-|x|^{2})(1-|y|^{2})}
\right).
\]
As a consequence, a subset of $\mathbb{H}^N$ is a hyperbolic ball if and only if it is a Euclidean ball in $\mathbb{R}^N$ contained in $\mathbb{H}^N$, possibly with a different center and radius, which can be computed explicitly from the above distance formula; see~\cite{RAT}.
We denote by $B_r(x)$ the geodesic ball of radius $r>0$ centered at $x \in \mathbb{H}^N$, that is,
\[
B_r(x) := \{\, y \in \mathbb{H}^N : d(x,y) < r \,\}.
\]

\medskip
\subsection{Sobolev and Poincar\'e inequality on the hyperbolic space.}

It is worth emphasizing that, on the hyperbolic space $\mathbb{H}^{N}$, the bottom of the spectrum of the nonlinear operator
\[
-\operatorname{div}\!\left(|\nabla_{\mathbb{H}} u|^{p-2}\nabla_{\mathbb{H}} u\right), \qquad p>1,
\]
is strictly positive. This stands in sharp contrast to the Euclidean setting, where the spectrum of the corresponding $p$-Laplacian coincides with $[0,\infty)$. More precisely, for every $p>1$, the Poincar\'e inequality on $\mathbb{H}^{N}$ asserts that
\begin{equation}\label{Poin}
\int_{\mathbb{H}^N} |\nabla_{\mathbb{H}} u|^p \, dV_{\mathbb{H}}
\geq \left(\frac{N-1}{p}\right)^p
\int_{\mathbb{H}^N} |u|^p \, dV_{\mathbb{H}},
\end{equation}
for all $u \in C_c^\infty(\mathbb{H}^N)$. The constant $\left(\frac{N-1}{p}\right)^p$ is sharp, although it is never attained by any nontrivial function. Furthermore, this inequality implies that the shifted operator
\[
-\operatorname{div}\!\left(|\nabla_{\mathbb{H}} u|^{p-2}\nabla_{\mathbb{H}} u\right)
- \left(\frac{N-1}{p}\right)^p
\]
is subcritical on $\mathbb{H}^{N}$, reflecting the strong influence of the negative curvature on the spectral and variational properties of the $p$-Laplacian. For further details, see \cite{BG16, BGG17, FLL22, FLLM23, KS16, MS08, NN19, H21}, among others.

\medskip

A substantial body of work has been devoted to refining the inequality~\eqref{Poin}. In particular, in the case $p=2$, for $N \geq 3,$ one obtains a Poincaré--Sobolev inequality of the form
\[
\int_{\mathbb{H}^N} |\nabla_{\mathbb{H}} u|^2 \, dV_{\mathbb{H}}
- \frac{(N-1)^2}{4} \int_{\mathbb{H}^N} |u|^2 \, dV_{\mathbb{H}}
\geq C(N,q)^2 \left( \int_{\mathbb{H}^N} |u|^q \, dV_{\mathbb{H}} \right)^{\frac{2}{q}},
\]
valid for all $u \in C_c^\infty(\mathbb{H}^N)$, where $2<q\leq \frac{2N}{N-2},$ and $C(N,q)>0$ is a constant depending only on $N$ and $q.$
In particular, at the critical Sobolev exponent $q=\frac{2N}{N-2}$, which corresponds to the Sobolev embedding on $\mathbb{H}^N$, the previous inequality reduces to the following critical Poincaré--Sobolev inequality:
\begin{equation}\label{PoinSob}
\int_{\mathbb{H}^N} |\nabla_{\mathbb{H}} u|^2 \, dV_{\mathbb{H}}
- \frac{(N-1)^2}{4} \int_{\mathbb{H}^N} |u|^2 \, dV_{\mathbb{H}}
\geq C(N)^2
\left( \int_{\mathbb{H}^N} |u|^{\frac{2N}{N-2}} \, dV_{\mathbb{H}} \right)^{\frac{N-2}{N}},
\end{equation}
where $C(N)>0$ is a constant depending only on the dimension. This inequality highlights the interplay between the spectral gap induced by the negative curvature of $\mathbb{H}^N$ and the critical Sobolev growth. 

In the three–dimensional case $N=3$, it was shown  by Benguria, Frank, and Loss~\cite{BFL08}, and by Mancini and Sandeep~\cite{MS08}  that the optimal constant in~\eqref{PoinSob} coincides with the sharp Euclidean Sobolev constant, namely
\[
C(N)=S(N,2).
\]
(see also Lu and Yang \cite{LY19} for a different proof.)
On the other hand, it was proved by Hebey \cite{Heb99} that for $N\ge 4$ we have $$C(N)<S(N, 2).$$
Here, for $1\leq p<N$, $S(N,p)$ denotes the best (sharp) constant in the Euclidean $L^p$–Sobolev inequality
\[
\int_{\mathbb{R}^N} |\nabla u|^p \, dx
\geq S(N,p)^p
\left( \int_{\mathbb{R}^N} |u|^{\frac{Np}{N-p}} \, dx \right)^{\frac{N-p}{N}},
\]
which holds for all $u\in C_c^\infty(\mathbb{R}^N)$. We refer to the classical works of Aubin~\cite{Aub76} and Talenti~\cite{T76} for the explicit characterization of $S(N,p)$ and the extremal functions. This identification of the optimal constant in the hyperbolic setting in dimension $N=3$ provides an explicit instance in which the geometry at infinity does not alter the sharp Sobolev constant. In particular, it furnishes a concrete realization, on $\mathbb{H}^3$, of the so-called \emph{$A$-part} of the $AB$ program on sharp Sobolev inequalities on Riemannian manifolds, as formulated by Druet and Hebey~\cite{DH02}.

More recently,   higher $k-th$ order Poincar\'e-Sobolev inequalities on hyperbolic spaces have been established by Lu and Yang in \cite{LY19} using the Helgason-Fourier analysis on symmetric spaces and sharp constants have been  shown to be equal to the $k-th$ order Sobolev constants in Euclidean space $\mathbb{R}^N$  for $N=\frac{2k+1}{2}$ and strictly smaller than  Sobolev constants for $N\ge 2k+2$ in \cite{LY221}. (see also related results on complex hyperbolic spaces by Lu and Yang \cite{LY22} and Flynn, Lu and Yang  on  on quaternionic hyperbolic spaces and on the Cayley hyperbolic plane \cite{FLY24}.)

\subsection{$AB$ Program.}
The so-called \emph{$A$-part} of the $AB$ program is concerned with establishing a sharp interpolation inequality that connects the critical Sobolev norm with the $L^p$ and gradient norms. More precisely, it aims to seek a constant $B$ for which there exists an exponent $\theta \in [1,p]$ such that the inequality
\begin{equation}\tag{$AB_{p,\mathrm{opt}}^{\theta}$}
S(N,p)^{\theta}\,\|u\|_{\frac{Np}{N-p}}^{\theta}
\leq \|\nabla u\|_p^{\theta} + B\,\|u\|_p^{\theta}
\end{equation}
holds for all $u \in C_c^\infty(\mathbb{R}^N)$. Here $S(N,p)$ denotes the sharp Euclidean Sobolev constant, and the inequality may be viewed as a refined Sobolev estimate in which the critical embedding is stabilized by the addition of a lower-order $L^p$ term. Determining the value of $B$—as well as the admissible range of $\theta$—is a central issue in understanding how geometric or spectral effects compensate for the lack of compactness at the critical exponent.
In the case of complete compact Riemannian manifolds, it was shown by Hebey and Vaugon~\cite{HV95, HV96}, Druet~\cite{Dru99}, and Aubin and Li~\cite{AL99} that~$(AB_{p, \mathrm{opt}}^{\theta})$ holds for $\theta = \min\{2, p\}$, resolving a longstanding conjecture due to Aubin~\cite{Aub76}. For a full exposition, see Aubin~\cite{Aub76}, Hebey~\cite{Heb99}, or Druet and Hebey~\cite{DH02}. For non-compact Riemannian manifolds, several results guarantee the validity of~$(AB_{p, \mathrm{opt}}^{\theta})$. For instance, Aubin, Druet, and Hebey~\cite{ADH98} proved that $(AB_{p, \mathrm{opt}}^{p})$ holds for any $1 \leq p < N$ with $B = 0$ on Cartan–Hadamard manifolds satisfying the Cartan–Hadamard conjecture. Specifically, $(AB_{p, \mathrm{opt}}^{p})$ is valid on the hyperbolic space for any $1 \leq p < N$. Since the inequality~\eqref{PoinSob} relates both to the sharp Poincaré and sharp Sobolev inequalities, the constants in~\eqref{PoinSob} are sharp and cannot be improved. Thus,~\eqref{PoinSob} provides an instance in which the sharp second constant $B$ can be computed explicitly; see~\cite{Heb99} for more examples when $p = 2$.

It is worth noting that, for dimensions $N \geq 4$, Hebey  \cite{Heb99}   proved that the constant $C(N)$ appearing in \eqref{PoinSob} is strictly smaller than the Euclidean Sobolev constant $S(N,2)$.  Tertikas and Tintarev   \cite{TT07} then proved the existence of the extremal functions of the inequality \eqref{PoinSob}. (see also a recent result of the existence of extremal function in $k-th$ higher order Poincar\'e-Sobolev inequalities by Lu and Tao \cite{LuTao}.) 
In this direction, by exploiting the conformal covariance of the GJMS operators on the hyperbolic space $\mathbb{H}^N$, Hebey established in \cite{Heb99} the following sharp Sobolev-type inequality: for all $u \in C_c^\infty(\mathbb{H}^N)$,
\begin{equation}\label{Sobolev}
\int_{\mathbb{H}^N} |\nabla_{\mathbb{H}} u|^2 \, dV_{\mathbb{H}}
- \frac{N(N-2)}{4} \int_{\mathbb{H}^N} |u|^2 \, dV_{\mathbb{H}}
\geq S(N,2)^2
\left( \int_{\mathbb{H}^N} |u|^{\frac{2N}{N-2}} \, dV_{\mathbb{H}} \right)^{\frac{N-2}{N}} .
\end{equation}
Here $S(N,2)$ denotes the optimal Sobolev constant in the Euclidean setting. Moreover, the term $\frac{N(N-2)}{4}$ is sharp for all $N \geq 4$, in the sense that it cannot be replaced by any larger constant while preserving the validity of the inequality. This result illustrates the subtle interplay between curvature effects, conformal invariance, and sharp functional inequalities on hyperbolic space. In fact, 
the left hand side is exactly  $\int_{\mathbb{H}^N} uP_1 udV_{\mathbb{H}}$, where $P_1$ is the conformal Laplacian on $\mathbb{H}^N$.
This Sobolev inequality can actually also be obtained by the conformal invariance of the Sobolev inequality on the Euclidean ball.


More recently, Nguyen \cite{H18} derived an $L^{p}$-analogue of the Poincaré–Sobolev inequality on the hyperbolic space $\mathbb{H}^{N}$. More precisely, for dimensions $N \geq 4$ with
\[
\frac{2N}{N-1} \leq p < N,
\]
Nguyen proved that the following inequality holds for all $u \in C_c^\infty(\mathbb{H}^N)$:
\begin{equation}\label{pPSob}
\int_{\mathbb{H}^N} |\nabla_{\mathbb{H}} u|^p \, dV_{\mathbb{H}}
- \left( \frac{N-1}{p} \right)^p \int_{\mathbb{H}^N} |u|^p \, dV_{\mathbb{H}}
\geq S(N,p)^p
\left( \int_{\mathbb{H}^N} |u|^{\frac{Np}{N-p}} \, dV_{\mathbb{H}} \right)^{\frac{N-p}{N}} .
\end{equation}
Here $S(N,p)$ denotes the optimal constant in the Euclidean $L^p$–Sobolev inequality. The inequality \eqref{pPSob} thus establishes the validity of the sharp Poincaré–Sobolev inequality corresponding to the optimal coefficient in front of the $L^{p}$-term, namely $(AB_{p,\mathrm{opt}}^{p})$. We note that when $p=N$, a Hardy-Moser-Trudinger inequality holds on $\mathbb{H}^N$ as shown by Liang et al in \cite{LLWY}.

Despite its strength, inequality \eqref{pPSob} does not cover all relevant cases. In particular, the three-dimensional setting $N=3$ is not included. Moreover, for dimensions $N \geq 4$, the range of exponents
\[
2 < p < \frac{2N}{N-1}
\]
remains open, leaving several natural and mathematically interesting regimes of the Poincaré-Sobolev inequality on $\mathbb{H}^{N}$ yet to be understood.

Our first main goal in this paper is to investigate $(AB_{p, \mathrm{opt}}^{p})$ with $B<0$ for the range $2 < p < \frac{2N}{N-1}$, thus extending  the validity of \eqref{pPSob} to the full  range $2<p<N$. More precisely, we will prove the following Poincar\'e-Sobolev inequality in the spirit of $(AB_{p, \mathrm{opt}}^{p})$ that complements the results in \cite{H18}:

\begin{theorem}\label{pSob}
Let $2 \leq p < N$. Then there exists a positive constant $\lambda(p,N)$ such that, for all $u \in C_c^\infty(\mathbb{H}^N)$,
\[
\int_{\mathbb{H}^N} |\nabla_{\mathbb{H}} u|^p \, dV_{\mathbb{H}}
- \lambda(p,N) \int_{\mathbb{H}^N} |u|^p \, dV_{\mathbb{H}}
\geq S(N,p)^p
\left( \int_{\mathbb{H}^N} |u|^{\frac{Np}{N-p}} \, dV_{\mathbb{H}} \right)^{\frac{N-p}{N}} .
\]
Here $S(N,p)$ denotes the sharp Euclidean $L^p$–Sobolev constant.
\end{theorem}

Moreover, as a consequence of our approach, we also obtain the Poincar\'e-Gagliardo-Nirenberg inequality in the spirit of $(AB_{p, \mathrm{opt}}^{p})$.

\medskip
\subsection{Logarithmic Sobolev inequalities.}
Our next main objective in this paper is to study the logarithmic Sobolev inequalities within the framework of the AB program, and to apply these results to investigate the Poincar\'e inequalities with Gaussian measures and their interpolations, commonly referred to as the Beckner inequalities. These inequalities are widely recognized as fundamental tools in functional analysis, probability theory, and mathematical physics, especially in the analysis of Markov processes, statistical mechanics, and information theory. They provide bounds on the entropy of a probability distribution relative to a reference measure, often yielding valuable insights into convergence rates in stochastic systems.

In the Euclidean setting, the logarithmic Sobolev inequality was first introduced by Leonard Gross in his seminal  work in 1975  \cite{Gro75}, in the framework of Gaussian measures. In its classical form, the inequality states that
\begin{equation}\label{GlogS}
\int_{\mathbb{R}^N} |u|^2 \ln |u|^2 \, d\mu
\leq 2 \int_{\mathbb{R}^N} |\nabla u|^2 \, d\mu ,
\end{equation}
where
\[
d\mu(x) = (2\pi)^{-\frac{N}{2}} e^{-\frac{|x|^2}{2}} \, dx
\]
denotes the standard Gaussian probability measure on $\mathbb{R}^N$, and the normalization condition
\[
\int_{\mathbb{R}^N} |u|^2 \, d\mu = 1
\]
is imposed. The constant $2$ in \eqref{GlogS} is sharp, and equality is achieved by functions of the form $u(x)=A e^{b\cdot x}$, with $A \in \mathbb{R}$ and $b \in \mathbb{R}^N$.

It is well known that the Gaussian logarithmic Sobolev inequality admits several equivalent formulations with respect to the Lebesgue measure. In particular, by a suitable change of variables, one can derive from \eqref{GlogS} the following Euclidean logarithmic Sobolev inequality:
\begin{equation}\label{ElogS}
\int_{\mathbb{R}^N} |u|^2 \ln |u| \, dx
\leq \frac{N}{4} \ln \!\left( \frac{2}{\pi N e}
\int_{\mathbb{R}^N} |\nabla u|^2 \, dx \right),
\end{equation}
under the normalization $\int_{\mathbb{R}^N} |u|^2 \, dx = 1$.
We refer, for instance, to \cite{Wei78} for further details and related formulations. Since their introduction, logarithmic Sobolev inequalities—both in the Gaussian and in the Lebesgue settings—have played a central role in analysis and probability theory. They have been extensively studied and have found numerous applications, ranging from functional inequalities and concentration of measure to partial differential equations and mathematical physics. See, for example, \cite{WB-99, MDG-04, YF, IG, LL18} and the references therein; this list is by no means exhaustive. In particular, the authors in \cite{PD03} derived the following optimal Euclidean $L^p$-Sobolev logarithmic inequality by using a family of Gagliardo-Nirenberg inequalities on $\mathbb{R}^N$ with full information about the sharp constants and optimizers and applying a suitable limiting process:
\begin{equation}\label{EplogS}
    \int_{\mathbb{R}^N}\left\vert u \right\vert^p\ln{\left\vert u \right\vert}dx \leq \frac{N}{p^2}\ln \left[\mathcal{L}_{N,p} \int_{\mathbb{R}^N}\left\vert \nabla u \right\vert^pdx\right]
\end{equation} with $\int_{\mathbb{R}^N}\left\vert u \right\vert^pdx=1$. Here the sharp constant is
\[
\mathcal{L}_{N,p} = \frac{p}{N}\pi^{-\frac{p}{2}} \left( \frac{p-1}{e} \right)^{p-1} \left( \frac{\Gamma\left( \frac{N}{2} + 1 \right) }{\Gamma\left( \frac{N(p-1)}{p} + 1 \right)} \right)^{\frac{p}{N}},
\]
with \(\Gamma\) denoting the Gamma function.

The second main objective of this paper is to study logarithmic Sobolev inequalities on the hyperbolic space $\mathbb{H}^N$ within the framework of the AB program. More specifically, we aim to determine whether there exists a constant $B$ such that the following logarithmic Sobolev inequality holds:
\begin{equation}\tag{$L_{p,\mathrm{opt}}$}
\int_{\mathbb{H}^N} |u|^p \ln |u| \, dV_{\mathbb{H}}
\leq \frac{N}{p^2}
\ln \!\left[
\mathcal{L}_{N,p}
\left(
\int_{\mathbb{H}^N} |\nabla_{\mathbb{H}} u|^p \, dV_{\mathbb{H}} - B
\right)
\right],
\qquad
\int_{\mathbb{H}^N} |u|^p \, dV_{\mathbb{H}} = 1 .
\end{equation}
The problem is to identify suitable $B\leq0$ under which this inequality is valid. This formulation mirrors the Euclidean theory while capturing the influence of the negative curvature of $\mathbb{H}^N$, and it provides a natural logarithmic counterpart to the $L^p-$Poincaré–Sobolev inequalities investigated earlier.

We remark that if the optimal constant $\mathcal{L}_{N,p}$ in $(L_{p, \mathrm{opt}})$ is not required, then logarithmic Sobolev inequalities can be directly obtained from Hölder's inequality. For instance, starting from Hölder's inequality on $\mathbb{H}^N$,
\[
\|u\|_q^q \leq \|u\|_p^{\alpha q} \|u\|_s^{(1-\alpha) q},
\]
where $1 \leq p \leq q \leq s < \infty$ and $\frac{\alpha q}{p} + \frac{(1-\alpha)q}{s} = 1$, or equivalently $\alpha = \frac{p}{q} \frac{s-q}{s-p}$, one may take the natural logarithm of both sides and deduce
\[
\ln\left(\frac{\|u\|_q}{\|u\|_p}\right) + (\alpha-1)\ln\left(\frac{\|u\|_s}{\|u\|_p}\right) \leq 0.
\]
It is clear that equality holds when $q = p$. Thus, if we define
\[
F(q):= \ln\left(\frac{\|u\|_q}{\|u\|_p}\right) + (\alpha-1)\ln\left(\frac{\|u\|_s}{\|u\|_p}\right),
\]
then $\left. \frac{dF}{dq} \right|_{q=p} \leq 0$. A direct computation shows that
\begin{align*}
\frac{dF}{dq}
&= \frac{d}{dq}\bigl[\ln \|u\|_q\bigr]
   + \frac{d}{dq}\!\left[(\alpha-1)\ln\!\left(\frac{\|u\|_s}{\|u\|_p}\right)\right] \\
&= \frac{d}{dq}\!\left[\frac{1}{q}\ln\!\left(\int_{\mathbb{H}^N} |u|^q \, dV_{\mathbb{H}}\right)\right]
   - \frac{ps}{q^2(s-p)} \ln\!\left(\frac{\|u\|_s}{\|u\|_p}\right) \\
&= -\frac{1}{q^2}\ln(\|u\|_q^q)
   + \frac{1}{q}\frac{1}{\|u\|_q^q}
     \int_{\mathbb{H}^N} |u|^q \ln |u| \, dV_{\mathbb{H}}
   - \frac{ps}{q^2(s-p)} \ln\!\left(\frac{\|u\|_s}{\|u\|_p}\right).
\end{align*}
Evaluating this expression at $q=p$ yields
\[
-\frac{1}{p}\ln \|u\|_p
+ \frac{1}{p}\frac{\displaystyle \int_{\mathbb{H}^N} |u|^p \ln |u| \, dV_{\mathbb{H}}}{\|u\|_p^p}
\leq \frac{s}{p(s-p)} \ln\!\left(\frac{\|u\|_s}{\|u\|_p}\right).
\]
Multiplying both sides by $\|u\|_p^p$ and rearranging terms, we obtain
\[
-\int_{\mathbb{H}^N} |u|^p \ln \|u\|_p \, dV_{\mathbb{H}}
+ \int_{\mathbb{H}^N} |u|^p \ln |u| \, dV_{\mathbb{H}}
\leq \frac{s}{s-p}\,\|u\|_p^p
\ln\!\left(\frac{\|u\|_s}{\|u\|_p}\right).
\]
Equivalently, this can be written in the compact form
\begin{equation}\label{1steqn}
\int_{\mathbb{H}^N}
\ln\!\left(\frac{|u|}{\|u\|_p}\right)
|u|^p \, dV_{\mathbb{H}}
\leq
\frac{s}{s-p}\,\|u\|_p^p
\ln\!\left(\frac{\|u\|_s}{\|u\|_p}\right),
\end{equation}
valid for all exponents $1 \leq p < s < \infty$. Assume henceforth that $\|u\|_p = 1$, with $1 \leq p < N$, and set
$s = \frac{Np}{N-p}.$ Applying inequality \eqref{1steqn} together with the Sobolev inequality on the hyperbolic space $\mathbb{H}^N$, we obtain
\begin{align}
\int_{\mathbb{H}^N} |u|^p \ln |u| \, dV_{\mathbb{H}}
&\leq \frac{N}{p} \ln \|u\|_s \nonumber \\
&\leq \frac{N}{p^2}
\ln\!\left(
\frac{1}{S(N,p)^p}
\int_{\mathbb{H}^N} |\nabla_{\mathbb{H}} u|^p \, dV_{\mathbb{H}}
\right).
\end{align}
Here $S(N,p)$ denotes the sharp constant in the Euclidean $L^p$–Sobolev inequality on $\mathbb{R}^N$. It is worth emphasizing that the constant $1/S(N,p)^p$ appearing above is not optimal in the logarithmic Sobolev setting. Indeed, as shown in \cite{PD03}, although $1/S(N,p)^p$ is asymptotically equivalent to the constant $\mathcal{L}_{N,p}$ as $N \to \infty$, it is strictly larger than $\mathcal{L}_{N,p}$ for every finite dimension $N$.
\medskip

In \cite{H18}, Nguyen proved a logarithmic $L^{p}$–version of the Poincaré–Sobolev inequality on the hyperbolic space $\mathbb{H}^{N}$. The result holds for $N \geq 4$ and $\frac{2N}{N-1} \leq p <N$. More precisely, under the normalization $\|u\|_{p}=1$, Nguyen showed that
\begin{equation}\label{logS1}
\int_{\mathbb{H}^N} |u|^{p} \ln |u| \, dV_{\mathbb{H}}
\leq \frac{N}{p^{2}}
\ln\!\left[
\mathcal{L}_{N,p}
\left(
\|\nabla_{\mathbb{H}} u\|_{p}^{p}
- \left(\frac{N-1}{p}\right)^{p}
\right)
\right].
\end{equation}
This inequality establishes the optimal logarithmic Sobolev inequality $(L_{p,\mathrm{opt}})$ in the above range of dimensions and exponents.
Nevertheless, several natural and mathematically significant cases remain beyond the scope of \eqref{logS1}. In particular, the three-dimensional case $N=3$ is not covered. Moreover, the endpoint case $p=2$—which is arguably the most fundamental and widely studied setting—remains open in this framework.

The investigation of logarithmic Sobolev-type inequalities on Riemannian manifolds is a relatively recent development. Nevertheless, it has long been understood that both the spectral gap and the logarithmic Sobolev constant admit quantitative lower and upper bounds in terms of geometric data of the manifold, such as the dimension, the diameter, and lower bounds on the Ricci curvature. For instance, Balogh, Krist{\'a}ly, and Tripaldi \cite{BKT} established sharp logarithmic Sobolev inequalities in the framework of metric measure spaces satisfying the curvature--dimension condition $CD(0,N)$. Using an $L^{1}$-optimal transport approach, Cavalletti and Mondino \cite{CM} proved the sharp L{\'e}vy--Gromov isoperimetric inequality, as well as sharp Sobolev and logarithmic Sobolev inequalities, on compact metric measure spaces satisfying the curvature--dimension condition $CD(K,N)$ with $K>0$. Notably, recent and influential contributions to the analysis of logarithmic Sobolev-type inequalities on Riemannian manifolds in noncompact and nonconvex settings have been made by Feng-Yu Wang \cite{FW-97,FW-09,FW-209}. We also refer the interested reader to \cite{Fang99, GGM05, MS14, FW-99}, among others, for additional results.

\medskip 

 In this paper, we investigate the logarithmic Sobolev--type inequality $(L_{p,\mathrm{opt}})$ for $p \geq 2$ within the framework of the AB program. This problem is subtle and appears to be new in the existing literature. The second main result of this paper addresses this question and provides new insights into the validity and optimality of logarithmic Sobolev inequalities in this setting. Our first result in this direction can be stated as follows:

\begin{theorem}\label{T1}
Let \( 2 \leq p < N \). There exists a constant \( \lambda(N, p) > 0 \) such that for any 
function \( u \in W^{1,p}(\mathbb{H}^N) \) with 
\( \|u\|_p = 1 \), the following inequality holds:
\[
\int_{\mathbb{H}^N} |u|^p \ln|u|\, dV_\mathbb{H}
\leq
\frac{N}{p^2} \ln \left[
\mathcal{L}_{N,p}
\left(
\int_{\mathbb{H}^N} |\nabla_\mathbb{H} u|^p\, dV_\mathbb{H}
-
\lambda(N, p)
\int_{\mathbb{H}^N} |u|^p\, dV_\mathbb{H}
\right)
\right].
\]

\end{theorem}

Moreover, for \( p = 2 \), one may choose 
\[
\lambda(N, 2) = \frac{N^2 (N-1)}{4(N+2)}.
\] 
Therefore, we have

\begin{theorem}\label{T1.1}
For any \( u \in W^{1,2}(\mathbb{H}^N) \) such that 
\( \|u\|_{L^2(\mathbb{H}^N)} = 1 \), it holds that
\[
\int_{\mathbb{H}^N} |u|^2 \ln |u| \, dV_{\mathbb{H}} 
\leq \frac{N}{4} \ln \left[ \frac{2}{\pi Ne} \left(
\int_{\mathbb{H}^N} |\nabla_{\mathbb{H}} u|^2 \, dV_{\mathbb{H}} 
- \frac{N^2 (N-1)}{4 (N+2)} \int_{\mathbb{H}^N} |u|^2 \, dV_{\mathbb{H}}
\right)
\right].
\]

\end{theorem}

In particular, when $N=3$, our result implies that for any \( u \in W^{1,2}(\mathbb{H}^3) \) such that 
\( \|u\|_{L^2(\mathbb{H}^3)} = 1 \), it holds that
\[
\int_{\mathbb{H}^N} |u|^2 \ln |u| \, dV_{\mathbb{H}} 
\leq \frac{3}{4} \ln \left[ \frac{2}{3\pi e}\int_{\mathbb{H}^3} uP_1 udV_{\mathbb{H}}
\right].
\]
Here $P_1$ is the conformal Laplacian on $\mathbb{H}^3$. Note that, unlike the power-type Sobolev case, the conformal covariance proof via GJMS operators does not extend to this logarithmic setting.

As mentioned earlier, in the Euclidean case, the \( L^2 \) logarithmic Sobolev inequality is equivalent to the Gaussian logarithmic Sobolev inequality. Following this philosophy, our next main purpose is to use the $L^2$-logarithmic Sobolev inequality (Theorem \ref{T1.1}) to establish a logarithmic Sobolev inequality with Gaussian measures on hyperbolic spaces. More precisely, we have
\begin{theorem}\label{T2}
    For any $u \in C^\infty_0(\mathbb{H}^N)$, we have
    \begin{align}\label{gnrGaussLogSob}
    &\int_{\mathbb{H}^N} u^2 \log(u^2)dm -\log\left(\int_{\mathbb{H}^N} u^2 dm\right)\int_{\mathbb{H}^N} u^2 dm \nonumber\\
    &\leq 2\int_{\mathbb{H}^N} |\nabla_{\mathbb{H}}u|^2 dm+\int_{\mathbb{H}^N} \left[(N-1)\left(\rho\coth \rho-1\right)-\dfrac{N^2(N-1)}{2(N+2)}+\log\left(C_2\right)\right]u^2dm,
\end{align}
where
\begin{equation}\label{C2}
    C_2:=(\mathcal{L}_{N,2})^{N/2}\dfrac{\sqrt{2 \pi}}{2^{N-1}}e^{\frac{(N-1)^2}{2}}\omega_{N-1} \left(\dfrac{Ne}{4}\right)^{N/2}.
\end{equation}
\end{theorem}

Here the weighted Gaussian measure \( dm \) is defined by
\[
dm = G^{-1} e^{-\frac{\rho^2}{2}} dV_{\mathbb{H}}, \quad 
G := \int_{\mathbb{H}^N} e^{-\frac{\rho^2}{2}} dV_{\mathbb{H}}, \quad 
\rho = \rho(x) = d_{\mathbb{H}}(x, 0) = \ln\left(\frac{1+|x|}{1 - |x|}\right),
\]
and  \(\omega_{N-1}\) is the area of the unit sphere $\mathbb{S}^{N-1} \subset \mathbb{R}^N$.

Notably, the constant $2$ on the right-hand side of \eqref{gnrGaussLogSob} is the sharp constant in the $L^2$-logarithmic Sobolev inequality for the standard Gaussian measure in Euclidean space. To our knowledge, \eqref{gnrGaussLogSob} is the first result achieving this Euclidean sharp constant for a logarithmic Sobolev inequality with Gaussian-type measure on hyperbolic space.


\medskip 

The paper is organized as follows.
\begin{itemize}

\item[Section 1:] The introduction offers a brief overview of Sobolev and logarithmic Sobolev inequalities in Euclidean and manifold settings, framed within the $AB$ program. We then present the paper's main results on hyperbolic space.

\item[Section 2:] We collect the key technical tools for the proofs, recalling symmetric decreasing rearrangements on $\mathbb{H}^N$ and establishing several lemmas that yield a non-trivial improvement of the P\'{o}lya--Szeg\"o inequality on hyperbolic space.

\item[Section 3:] We prove the Poincar\'e--Sobolev, the Poincar\'e--Gagliardo--Nirenberg and the Poincar\'e--logarithmic Sobolev inequalities  on $\mathbb{H}^N$, with detailed proofs of Theorem~\ref{pSob} and Theorem~\ref{T1}.

\item[Section 4:] This section addresses Gaussian logarithmic Sobolev inequalities, proving Theorem~\ref{T2}, and studies generalized Poincar\'{e} inequalities on $\mathbb{H}^N$ with Gaussian-type measures.

\item[Section 5:] Motivated by the fact that the Beckner inequality interpolates between the Poincar\'e and Gross logarithmic Sobolev inequalities, we establish an extended Beckner inequality on $\mathbb{H}^N$ with a suitably modified measure.

\item[Section 6:] Finally, we extend key results—such as improved P\'{o}lya--Szeg\"o and logarithmic Sobolev inequalities—to Riemannian model manifolds satisfying the centered isoperimetric inequality.

\end{itemize}

\medskip


\section{Symmetric decreasing rearrangements: Improvement of P\'{o}lya--Szeg\"o principle}\label{sec-2}

This section begins by recalling the notion of rearrangements on the hyperbolic space $\mathbb{H}^N$ and the fundamental measure-theoretic principles underlying symmetrization. We then review the classical P\'{o}lya--Szeg\"o inequality on $\mathbb{H}^N$, which plays a central role in the analysis of functional inequalities on curved spaces. Building on this framework, and through a series of delicate and careful computations, we establish a refined version of the P\'{o}lya--Szeg\"o inequality on the hyperbolic space, yielding an improvement over the classical result.
\medskip

Let $u:\mathbb{H}^N \rightarrow \mathbb{R}$ be a function such that

$$\Vol_g\left(\{x \in \mathbb{H}^N:\left\vert u(x)\right\vert>t\}    \right) = \int_{\{x \in \mathbb{H}^N:\left\vert u(x)\right\vert>t\}} dV_{\mathbb{H}} < \infty,~~~\forall t>0. $$
For such a function $u$, its distribution function, denoted by $\mu_u$, is defined by
$$\mu_u \left( t\right)=\Vol_g\left(\{x \in \mathbb{H}^N:\left\vert u(x)\right\vert>t\}    \right),~~~\forall t>0.$$
The above function is non-increasing and right-continuous. Then the decreasing rearrangement
function $u^*$ of $u$ is defined by

$$u^*\left( t\right)=\sup\{s>0:\mu_u\left( s\right)>t\}.$$
It is easy to see that the function $u^*\left( t\right)$ is non-increasing. We then define the symmetric decreasing
rearrangement function $u_g^{\sharp}$ of $u$ by

$$u_g^{\sharp}\left( x\right)=u^*\left( \Vol_g\left(B_g\left(0,\rho\left(x\right)\right)\right)\right),~~~x\in \mathbb{H}^N$$
and the Euclidean symmetric decreasing
rearrangement function $u_e^{\sharp}$ of $u$ by

$$u_e^{\sharp}\left( x\right)=u^*\left( \sigma_N \left\vert x\right\vert^N\right),~~~x\in \mathbb{R}^N.$$
Then, we have by the layer cake representation that for any nondecreasing function $F:\left[0,\infty \right) \rightarrow \left[0,\infty \right)$ with $F\left(0\right)=0$:

$$\int_{\mathbb{H}^N}F\left(\left\vert u \right\vert \right)dV_{\mathbb{H}}=\int_{\mathbb{H}^N}F\left(u_g^{\sharp}\right)dV_{\mathbb{H}}=\int_{\mathbb{R}^N}F\left(u_e^{\sharp}\right)dx=\int_{0}^{\infty}F\left(u^{*}\left(t\right)  \right)dt.$$

\subsection{P\'{o}lya-Szeg\"o principle and its improvement}

\medskip

 The classical P\'{o}lya--Szeg\"o principle on $\mathbb{H}^N$, the symmetric decreasing rearrangement $u_g^\sharp$ satisfies
\[
\int_{\mathbb{H}^N} \bigl|\nabla_g u_g^\sharp\bigr|_g^p \, dV_{\mathbb{H}}
\;\le\;
\int_{\mathbb{H}^N} \bigl|\nabla_g u\bigr|_g^p \, dV_{\mathbb{H}} .
\]
Moreover, a direct computation (see \cite{H18}) yields the precise decomposition
\begin{equation}\label{SRI}
\int_{\mathbb{H}^N} \bigl|\nabla_g u_g^\sharp\bigr|_g^p \, dV_{\mathbb{H}}
=
\int_{\mathbb{R}^N} \bigl|\nabla u_e^\sharp\bigr|^p \, dx
+
\bigl(N\sigma_N\bigr)^p
\int_0^\infty \bigl|v'(s)\bigr|^p \,
k_{N,p}\!\left(\frac{s}{\sigma_N}\right) ds ,
\end{equation}
where $v=u^*$ is the decreasing rearrangement and
\[
k_{N,p}(s)
:=
\bigl(\sinh \Phi^{-1}(s)\bigr)^{p(N-1)}
-
s^{\frac{p(N-1)}{N}},
\qquad
\Phi(t)
:=
N \int_0^t (\sinh y)^{N-1} dy ,
\]
with
\[
s = \Vol_g\!\bigl(B_g(0,t)\bigr) = \sigma_N \Phi(t).
\]

\medskip 

Assume that $p \ge 2$. We prove subsequently that there exists a constant $C(N,p)>0$ such that
\[
k_{N,p}(s) \ge C(N,p)\, s^p
\qquad \text{for all } s \ge 0 .
\]
This lower bound implies the existence of a constant $\lambda(N,p)>0$ with the property that, for every $N \ge 3$ and for all admissible functions $u$, the following inequality holds:
\[
\int_{\mathbb{H}^N} \bigl|\nabla_g u_g^\sharp\bigr|_g^p \, dV_{\mathbb{H}}
-
\lambda(N,p)
\int_{\mathbb{H}^N} |u_g^\sharp|^p \, dV_{\mathbb{H}}
\;\ge\;
\int_{\mathbb{R}^N} \bigl|\nabla u_e^\sharp\bigr|^p \, dx .
\]


To this end, we have the following lemmas:

\medskip

\begin{lemma}\label{L2I} We have
    $$k_{N,2}(\Phi(t)) \geq \dfrac{N-1}{N+2} (\Phi(t))^2,$$
for all $N \geq 3, t \geq 0$.
\end{lemma}
 Here, 
$k_{N,2}(s):=(\sinh\Phi^{-1}(s))^{2(N-1)}-s^{2(N-1)/N}$, and $\Phi(t):=N\int_0^t (\sinh s)^{N-1}ds$.
\begin{proof}
 Since $\Phi$ is a diffeomorphism, it suffices to determine $C(N)$ such that 
$$F_{N,2}(t)=k_{N,2}(\Phi(t))-C(N)(\Phi(t))^2=(\sinh t)^{2(N-1)}-\Phi(t)^{\frac{2(N-1)}{N}}-C(N)(\Phi(t))^2 \geq 0,$$
for all $t \geq 0$ and $N \geq 3$. To do that, we would like to compute
$$\displaystyle \lim_{t \rightarrow 0} \dfrac{k_{N, 2}(\Phi(t))}{\Phi^2(t)}.$$
By using the L'Hospital rule, we have
\begin{align*}
    \displaystyle \lim_{t \rightarrow 0} \dfrac{k_{N, 2}(\Phi(t))}{\Phi^2(t)}=\lim_{t \rightarrow 0} \dfrac{\left(k_{N, 2}(\Phi(t))\right)^{\prime}}{2\Phi(t) \Phi'(t)}=\dfrac{N-1}{N}\lim_{t\rightarrow 0}\dfrac{\sinh^{N-2}(t)\cosh t-(\Phi(t))^{(N-2)/N}}{\Phi(t)}.
\end{align*}
Applying the L'Hospital rule once again, we obtain
\begin{align*}
    \displaystyle \lim_{t \rightarrow 0} \dfrac{k_{N, 2}(\Phi(t))}{\Phi^2(t)}&=\dfrac{N-1}{N}\lim_{t\rightarrow 0}\dfrac{\sinh^{N-2}(t)\cosh t-(\Phi(t))^{(N-2)/N}}{\Phi(t)}\\
    &=\dfrac{(N-1)(N-2)}{N^2}\lim_{t\rightarrow 0}\left[(\coth t)^2+\dfrac{1}{N-2}-(\Phi(t))^{-2/N}\right].
\end{align*}
We have $$\coth (t)=\dfrac{1}{t}+\dfrac{1}{3}t-\dfrac{1}{45}t^3+o(t^3)=\dfrac{1}{t}(1+\dfrac{1}{3}t^2+o(t^2)),\;\text{as $t \rightarrow 0^+$},$$
then 
$$(\coth(t))^2=\dfrac{1}{t^2}\left(1+\dfrac{1}{3}t^2+o(t^2)\right)^2=\dfrac{1}{t^2}\left(1+\dfrac{2}{3}t^2+o(t^2)\right)=\dfrac{1}{t^2}+\dfrac{2}{3}+o(1),\;\text{as $t \rightarrow 0^+$}.$$
Moreover,
\begin{align*}
    \Phi(t)=N \int_0^t (\sinh s)^{N-1}ds&=N \int_0^t \left(s+\dfrac{s^3}{3!}+\dfrac{s^5}{5!}+o(s^5)\right)^{N-1}ds\\
    &=N \int_0^t s^{N-1}\left(1+\dfrac{s^2}{3!}+\dfrac{s^4}{5!}+o(s^4)\right)^{N-1}ds\\
    &=N \int_0^t s^{N-1} \left(1+\dfrac{(N-1)}{3!}s^2+o(s^2)\right)ds\\
    &=t^N+\dfrac{N(N-1)}{3!}\dfrac{1}{N+2}t^{N+2}+o(t^{N+2}),\;\text{as $t \rightarrow 0^+$},
\end{align*}
which implies that
\begin{align*}
    (\Phi(t))^{-2/N}=t^{-2}\left(1+\dfrac{N(N-1)}{3!(N+2)}t^2+o(t^2)\right)^{-2/N}&=t^{-2}\left(1-\dfrac{2}{N}\dfrac{N(N-1)}{3!(N+2)}t^2+o(t^2)\right)\\
    &=t^{-2}-\dfrac{1}{3}\dfrac{N-1}{N+2}+o(1),
\end{align*}
as $t \rightarrow 0^+$. Thus,
$$\lim_{t \rightarrow 0} \dfrac{k_{N, 2}(\Phi(t))}{\Phi^2(t)}=\dfrac{(N-1)(N-2)}{N^2} \dfrac{N^2}{(N-2)(N+2)}=\dfrac{N-1}{N+2}.$$

Hence, next, we would like to prove that
$$k_{N,2}(\Phi(t)) \geq \dfrac{N-1}{N+2} (\Phi(t))^2,$$
for all $N \geq 3, t \geq 0$, which is equivalent to $F_{N,2}(t) \geq 0$.

By a direct computation, we have
\begin{align*}
    F'_{N,2}(t)&=2(N-1)(\sinh t)^{2N-3}\cosh t\\
    &-2(N-1) (\Phi(t))^{(N-2)/N}(\sinh t)^{N-1}-2\dfrac{N(N-1)}{N+2}\Phi(t)(\sinh t)^{N-1}\\
    &=2(N-1)(\sinh t)^{N-1}\left[(\sinh t)^{N-2}\cosh t-(\Phi(t))^{(N-2)/N}-\dfrac{N}{N+2}\Phi(t)\right]\\
    &:=2(N-1)(\sinh t)^{N-1} G(t).
\end{align*}
Then, 
\begin{align*}
    G'(t)&=(N-2)(\sinh t)^{N-3}(\cosh t)^2+(\sinh t)^{N-1}\\
    &-(N-2)(\Phi(t))^{-2/N}(\sinh t)^{N-1}-\dfrac{N^2}{N+2}(\sinh t)^{N-1}\\
    &=\dfrac{N-2}{N+2}(\sinh t)^{N-1}+(N-2)(\sinh t)^{N-3}-(N-2)(\Phi(t))^{-2/N}(\sinh t)^{N-1}.
\end{align*}
We claim that $G'(t) \geq 0$ for all $t \geq 0$. If so, $G(t) \geq G(0)=0$, which implies that $F'_{N,2}(t) \geq 0$ and $F_{N,2}(t) \geq F_{N,2}(0)=0$. In fact, it is enough to prove
$$\Phi(t) \geq \left(\dfrac{(\sinh t)^2}{\frac{(\sinh t)^2}{N+2}+1}\right)^{N/2}.$$
Define $H(t):=\Phi(t)-\left(\dfrac{(\sinh t)^2}{\frac{(\sinh t)^2}{N+2}+1}\right)^{N/2}$. Thus,
\begin{align*}
    H'(t)&=N(\sinh t)^{N-1}-\dfrac{N}{2}\dfrac{2\sinh t \cosh t}{\left(\frac{(\sinh t)^2}{N+2}+1\right)^{(N+2)/2}}(\sinh t)^{N-2}\\
    &=N(\sinh t)^{N-1}-N\dfrac{\cosh t}{(\sinh t)^3}\left(\dfrac{(\sinh t)^2}{\frac{(\sinh t)^2}{N+2}+1}\right)^{(N+2)/2}.
\end{align*}
Hence, $H'(t) \geq 0$ iff 
$$(\sinh t)^{N+2} \geq \cosh t \left(\dfrac{(\sinh t)^2}{\frac{(\sinh t)^2}{N+2}+1}\right)^{(N+2)/2},$$
which is equivalent to
$$\left(\frac{(\sinh t)^2}{N+2}+1\right)^{N+2} \geq (\cosh t)^2.$$
The above inequality is obvious by the Bernoulli inequality which states that $\left(1+x\right)^r\geq 1+rx$ for every real number $r \geq 1$ and $x \geq -1$. Then, $H(t)\geq H(0)=0$ or $G'(t) \geq 0$, as desired. Thus,
$$k_{N,2}(s) \geq \dfrac{N-1}{N+2} s^2,$$
for all $N \geq 3$, and $s \geq 0$.
\end{proof}

\begin{lemma}\label{LpI} For $p > 2$. There exists $C(N,p)>0$ such that $$k_{N,p}(s) \geq C(N,p) s^p$$
for all $s\geq 0$.
    
\end{lemma}
\begin{proof}
    Since $\Phi$ is a diffeomorphism, it suffices to show that $$\displaystyle \lim_{t \rightarrow 0, \infty} \dfrac{k_{N, p}(\Phi(t))}{\Phi^p(t)} > 0.$$
    Note that $\Phi (t)=N\int_0^t (\sinh s)^{N-1}ds \rightarrow \infty$ as $t \rightarrow \infty$, we have

   \begin{align*}
     \displaystyle \lim_{t \rightarrow \infty} \dfrac{k_{N, p}(\Phi(t))}{\Phi^p(t)} &= \displaystyle \lim_{t \rightarrow \infty} \dfrac{\sinh^{p(N-1)} t -\Phi^{\frac{p(N-1)}{N}}(t) }{\Phi^p(t)}
     =\displaystyle \lim_{t \rightarrow \infty} \dfrac{\sinh^{p(N-1)} t}{\Phi^p(t)}
      =\left(\displaystyle \lim_{t \rightarrow \infty} \dfrac{\sinh^{(N-1)} t}{\Phi(t)}\right)^p\\
     &=\left(\displaystyle \lim_{t \rightarrow \infty} \dfrac{(N-1)\sinh^{(N-2)} t \cosh t}{N\sinh ^{N-1}t}\right)^p
     =\left( \frac{N-1}{N}\right)^p.
   \end{align*}
   
We also compute
\[
\lim_{t \rightarrow 0} \frac{\sinh^{p(N-1)}(t)-\Phi^{\frac{p(N-1)}{N}}(t)}{\Phi^p(t)}.
\]
Applying L'Hospital's rule twice, it is straightforward to get
\begin{align*}
    &\lim_{t \rightarrow 0} \frac{\sinh^{p(N-1)}(t)-\Phi^{\frac{p(N-1)}{N}}(t)}{\Phi^p(t)}\\
    &=\frac{(N-1)((p-1)N-p)}{(p-1)N^2}\\
    &\times\lim_{t\rightarrow 0} \left[\frac{(\sinh(t))^{(p-2)N-p}(\sinh'(t))^2}{(\Phi(t))^{p-2}}+\frac{1}{(p-1)N-p}\frac{(\sinh(t))^{(p-2)N+1-p}\sinh''(t)}{(\Phi(t))^{p-2}}-\Phi^{-p/N}(t)\right].
\end{align*}
Note that $\sinh(t) = t + \frac{1}{6} t^3 + o(t^3)$ as $t \rightarrow 0$, we have
\[
\sinh(t)^{(p-2)N-p} = t^{(p-2)N-p}\left(1+((p-2)N-p)\frac{1}{6}t^2+o(t^2)\right),
\]
\[
\sinh(t)^{(p-2)N-p+1} = t^{(p-2)N-p+1}\left(1+((p-2)N-p+1)\frac{1}{6} t^2 + o(t^2)\right),
\]
and
\[
(\sinh'(t))^2 = \left(\cosh(t)\right)^2 = \left(1 + \frac{t^2}{2} + o(t^2)\right)^2 = 1 + t^2 + o(t^2).
\]
Since $\Phi(t)=t^N+\frac{N(N-1)}{N+2}\frac{1}{6}t^{N+2}+o(t^{N+2})$, we have
\[
\Phi(t)^{2-p} = t^{N(2-p)}\left(1+(2-p)\frac{N(N-1)}{N+2}\frac{1}{6}t^2+o(t^2)\right),
\]
and
\[
\Phi(t)^{-p/N}=t^{-p}\left(1-\frac{p(N-1)}{N+2}\frac{1}{6}t^2+o(t^2)\right).
\]
Then,
\begin{align*}
    &\frac{(\sinh(t))^{(p-2)N-p}(\sinh'(t))^2}{(\Phi(t))^{p-2}} + \frac{1}{(p-1)N-p}\frac{(\sinh(t))^{(p-2)N+1-p}\sinh''(t)}{(\Phi(t))^{p-2}} - \Phi^{-p/N}(t) \\
    &= \left((p-2)\frac{3N}{N+2}+6-p+\frac{6}{(p-1)N-p}+\frac{p(N-1)}{N+2}\right) \frac{1}{6} t^{2-p} + o(t^{2-p}),
\end{align*}
which tends to $\infty$ as $p > 2$.
\end{proof}
\medskip

The above two Lemmas led us to the following version of P\'{o}lya--Szeg\"o inequality.

\begin{lemma}\label{keyLem}
    For $p \geq 2$ and $N \geq 3$, there exists a constant $\lambda(N,p)>0$ such that
    \begin{equation}\label{key1}        
   \|\nabla_{\mathbb{H}}u\|^p_p-\lambda(N,p)\|u\|^p_p \geq \|\nabla u_e^\sharp\|^p_p.
    \end{equation}
    In particular,
    \begin{equation}\label{key2} 
    \|\nabla_{\mathbb{H}}u\|^2_2-\dfrac{N^2(N-1)}{4(N+2)}\|u\|^2_2 \geq \|\nabla u_e^\sharp\|^2_2.
    \end{equation}
\end{lemma}

\begin{proof}
We have by \eqref{SRI} and Lemma \ref{L2I} that
\begin{align*}
    \int_{\mathbb{H}^N}|\nabla_{\mathbb{H}}u_g^\sharp|^2_gdV_{\mathbb{H}}&=\int_{\mathbb{R}^N}|\nabla u_e^\sharp|^2dx+(N\sigma_N)^2 \int_0^\infty |v'(s)|^2 k_{N,2}\left(\frac{s}{\sigma_N}\right)ds\\
    &\geq \int_{\mathbb{R}^N}|\nabla u_e^\sharp|^2dx+\dfrac{N^2(N-1)}{N+2}\int_0^\infty |v'(s)|^2 s^2ds,
\end{align*}
where $v:=u^*$. Define $w(s):=v(s)s^{1/2}$, then
$$\int_0^\infty |v'(s)|^2s^2ds=\int_0^\infty |w'(s)|^2s+\dfrac{1}{4}\int_0^\infty (v(s))^2ds,$$
and 
$$\int_{\mathbb{H}^N}|\nabla_{\mathbb{H}} u_g^\sharp|^2_gdV_{\mathbb{H}} \geq \int_{\mathbb{R}^N}|\nabla u_e^\sharp|^2dx +\dfrac{N^2(N-1)}{N+2}\int_0^\infty |w'(s)|^2s+\dfrac{N^2(N-1)}{4(N+2)}\int_0^\infty (v(s))^2ds.$$
Using $$\int_0^\infty (v(s))^2ds=\int_{\mathbb{H}^N}|u_g^\sharp|^2dV_{\mathbb{H}},$$
we have
$$\int_{\mathbb{H}^N}|\nabla_{\mathbb{H}}u_g^\sharp|_g^2dV_{\mathbb{H}}-\dfrac{N^2(N-1)}{4(N+2)} \int_{\mathbb{H}^N} |u_g^\sharp|^2_gdV_{\mathbb{H}}\geq \int_{\mathbb{R}^N} |\nabla u_e^\sharp|^2dx.$$
Hence, applying P\'{o}lya-Szeg\"{o} inequality, we get the desired estimate \eqref{key2}.

Similarly, by applying \eqref{SRI} and Lemma \ref{LpI} with the definition \(w(s):=v(s)s^{1/p}\), we can also derive \eqref{key1}.

\end{proof}

\medskip

\section{Poincar\'{e}-Sobolev, Poincar\'{e}-Gagliardo-Nirenberg and Poincar\'e-logarithmic Sobolev inequalities on $\mathbb{H}^N$-Proofs of Theorem \ref{pSob} and Theorem \ref{T1}}

In this section, we apply Lemma~\ref{keyLem} to establish Poincaré–Sobolev, Poincaré-Gagliardo-Nirenberg and Poincar\'e-logarithmic Sobolev inequalities on the hyperbolic space $\mathbb{H}^N$. More precisely, we show how the framework provided by Lemma~\ref{keyLem} yields these functional inequalities in a natural and unified manner in the hyperbolic setting. We begin the proof of Theorem~\ref{pSob}

  \begin{proof}[Proof of Theorem \ref{pSob}]
By Lemma~\ref{keyLem}, the hyperbolic gradient term can be compared with its Euclidean counterpart via the symmetric decreasing rearrangement. Combining this with the sharp $L^p$–Sobolev inequality on $\mathbb{R}^N$, we obtain
\begin{align*}
\|\nabla_{\mathbb{H}} u\|_{p}^{p}-\lambda(N,p)\|u\|_{p}^{p}
&\geq \|\nabla u_e^\sharp\|_{p}^{p} \\
&\geq S(N,p)^p\,\|u_e^\sharp\|_{\frac{Np}{N-p}}^{p}.
\end{align*}
Finally, since rearrangement preserves $L^q$ norms, we have
\[
\|u_e^\sharp\|_{\frac{Np}{N-p}}=\|u\|_{\frac{Np}{N-p}},
\]
which yields
\[
\|\nabla_{\mathbb{H}} u\|_{p}^{p}-\lambda(N,p)\|u\|_{p}^{p}
\geq S(N,p)^p\|u\|_{\frac{Np}{N-p}}^{p}.
\]
This completes the proof.
\end{proof}

 \medskip   

Next, in the same spirit, we derive a family of Poincaré–Gagliardo–Nirenberg inequalities on the hyperbolic space $\mathbb{H}^N$. As a preliminary step, we recall the sharp Gagliardo–Nirenberg inequalities in the Euclidean space $\mathbb{R}^N$, which were established by Del Pino and Dolbeault in \cite{PD02, PD03} (see also Lam and Lu \cite{LL17} for the weighted version). These inequalities play a fundamental role in our analysis, as they provide the optimal interpolation framework needed to transfer sharp estimates from the Euclidean setting to the hyperbolic one via the rearrangement and comparison principles developed earlier.

\begin{lemma}
    Let $1<p<N$, $\alpha\in \left(0, \frac{N}{N-p}\right]$, $\alpha \neq 1$. Then, for all $u \in C^\infty_0(\mathbb{R}^N)$
    \begin{itemize}
        \item [(i)] for $\alpha > 1$,
\[
\|u\|_{L^{\alpha p}(\mathbb{R}^N)} \leq GN_1(N, p, \alpha) \| \nabla u \|_{L^p(\mathbb{R}^N)}^{\theta} \| u \|_{L^{\alpha(p-1)+1}(\mathbb{R}^N)}^{1-\theta}
\]
with
\[
\theta = \frac{N(\alpha - 1)}{\alpha (N p - (\alpha p + 1 - \alpha)(N - p))},
\]
the sharp constant $GN_1(N, p, \alpha)$ is given by
\[
GN_1(N, p, \alpha) = 
\left( \frac{q-p}{p\sqrt{\pi}} \right)^{\theta}
\left( \frac{p q}{N(q-p)} \right)^{\frac{\theta}{p}}
\left( \frac{\delta}{p q} \right)^{\frac{1}{\alpha p}}
\left(
  \frac{
    \Gamma\left( q \frac{p-1}{q-p} \right)
    \Gamma\left( \frac{N}{2} + 1 \right)
  }{
    \Gamma\left( \frac{p-1}{p}\frac{\delta}{q-p} \right)
    \Gamma\left( N \frac{p-1}{p} + 1 \right)
  }
\right)^{\frac{\theta}{N}}
\]
with
\[
q = \alpha(p-1)+1, \quad \delta = Np-(N-p)q,
\]
and an extremal function is given by
\[
u(x) = \left( 1 + |x|^{\frac{p}{p-1}} \right)^{-\frac{1}{\alpha-1}}.
\]
        \item [(ii)] for $0<\alpha < 1$,
\[
\|u\|_{L^{\alpha(p-1)+1}(\mathbb{R}^N)} \leq GN_2(N, p, \alpha) \| \nabla u \|_{L^p(\mathbb{R}^N)}^{\theta} \| u \|_{L^{\alpha p}(\mathbb{R}^N)}^{1-\theta},
\qquad u \in C_0^\infty(\mathbb{R}^N),
\]
with
\[
\theta = \frac{N(1-\alpha)}{(\alpha p+1-\alpha)(N-\alpha(N-p))},
\]
the sharp constant $GN_2(N, p, \alpha)$ is given by
\[
GN_2(N, p, \alpha) =
\left( \frac{p-q}{p\sqrt{\pi}} \right)^\theta
\left( \frac{p q}{N(p-q)} \right)^{\frac{\theta}{p}}
\left( \frac{p q}{\delta} \right)^{\frac{1-\theta}{\alpha p}}
\left(
  \frac{
    \Gamma\left( \frac{p-1}{p}\frac{\delta}{p-q} + 1 \right)
    \Gamma\left( \frac{N}{2} + 1 \right)
  }{
    \Gamma\left( q\frac{p-1}{p-q}+1 \right)
    \Gamma\left( N \frac{p-1}{p} + 1 \right)
  }
\right)^{\frac{\theta}{N}}
\]
with
\[
q = \alpha(p-1)+1, \qquad \delta = Np - q(N-p) > 0,
\]
and an extremal function is given by
\[
u(x) = \left( 1 - |x|^{\frac{p}{p-1}} \right)_+^{\frac{1}{1-\alpha}},\qquad a_+ = \max\{a, 0\}.
\]
    \end{itemize}
\end{lemma}
As a direct consequence of this lemma and our lemma \ref{keyLem}, we get 
\begin{theorem}
    Let $2 \leq p<N$, $\alpha\in \left(0, \frac{N}{N-p}\right]$, $\alpha \neq 1$. Then, there exists $\lambda(N,p)>0$ (with $\lambda(N,2)=\dfrac{N^2(N-1)}{4(N+2)}$) such that for all $u \in C^\infty_0(\mathbb{H}^N)$, we have
    \begin{itemize}
        \item [(i)] for $\alpha>1$,
        \[
\|u\|_{\alpha p} \leq GN_1(N, p, \alpha) \left( \|\nabla_{\mathbb{H}}u\|^p_p-\lambda(N,p)\|u\|^p_p \right)^{\frac{\theta}{p}} \| u \|_{\alpha(p-1)+1}^{1-\theta}
\]
with
\[
\theta = \frac{N(\alpha - 1)}{\alpha (N p - (\alpha p + 1 - \alpha)(N - p))},
\]
        \item [(ii)] for $0<\alpha<1$,
        \[
\|u\|_{\alpha(p-1)+1} \leq GN_2(N, p, \alpha) \left( \|\nabla_{\mathbb{H}}u\|^p_p-\lambda(N,p)\|u\|^p_p \right)^{\frac{\theta}{p}} u \|_{\alpha p}^{1-\theta}
\]
with
\[
\theta = \frac{N(1-\alpha)}{(\alpha p+1-\alpha)(N-\alpha(N-p))}.
\]
    \end{itemize}
\end{theorem}
\begin{proof}
   The proof is similar to the proof of Theorem~\ref{pSob} and is left to the reader.

\end{proof}
\begin{proof}[Proof of Theorem \ref{T1}]
    First, we claim that for all $p \geq 2$,
$$\int_{\mathbb{H}^N}|u|^p\ln |u| dV_{\mathbb{H}}=\int_{\mathbb{R}^N}(u_e^\sharp)^p\log (u_e^\sharp) dx.$$
Indeed, 
$$\int_{\mathbb{H}^N}|u|^p\log |u| dV_{\mathbb{H}}=\int_{\{|u|<1\}}|u|^p\log |u| dV_{\mathbb{H}}+\int_{\{|u|\geq1\}}|u|^p\log |u| dV_{\mathbb{H}}<\infty.$$
We have
$$\int_{\{|u|<1\}}|u|^p\log |u| dV_{\mathbb{H}}=-\int_{\{|u|<1\}}|u|^p|\log |u|| dV_{\mathbb{H}}=\int_0^1 V(\{|u|>t\})\left(pt^{p-1}\log t+t^{p-1}\right)dt.$$
Indeed,
\begin{align*}
   & \int_0^1 V(\{|u|>t\})\left(pt^{p-1}\log t+t^{p-1}\right)dt \;=\;\int_0^1 \left(\int_{\mathbb{H}^N}\chi_{\{|u|>t\}}(x)dV_{\mathbb{H}}\right)\left(pt^{p-1}\log t+t^{p-1}\right)dt\\
    &=\int_{\mathbb{H}^N} \left(\int_0^1 \chi_{\{|u|>t\}}(x)\left(pt^{p-1}\log t+t^{p-1}\right)dt\right)dV_{\mathbb{H}}
    =\int_{\mathbb{H}^N} \left(\int_0^{1\wedge |u(x)|} \left(pt^{p-1}\log t+t^{p-1}\right)dt\right)dV_{\mathbb{H}}\\
    &=\int_{\{|u|<1\}}\left(\int_0^{|u(x)|} \left(pt^{p-1}\log t+t^{p-1}\right)dt\right)dV_{\mathbb{H}}
    +\int_{\{|u| \geq1\}}\left(\int_0^{1} \left(pt^{p-1}\log t+t^{p-1}\right)dt\right)dV_{\mathbb{H}}\\
    &=\int_{\{|u|<1\}}\left(\int_0^{|u(x)|} \left(pt^{p-1}\log t+t^{p-1}\right)dt\right)dV_{\mathbb{H}}\\
    &=\int_{\{|u|<1\}} |u|^p \log|u|dV_{\mathbb{H}}.
\end{align*}
Here, we used $\int_0^1 \left(pt^{p-1}\log t+t^{p-1}\right)dt=0$. Similarly, 
$$\int_{\{|u| \geq 1\}}|u|^p\log |u| dV_{\mathbb{H}}=\int_1^\infty V(\{|u|>t\})\left(pt^{p-1}\log t+t^{p-1}\right)dt.$$
More clearly, 
\begin{align*}
   & \int_1^\infty V(\{|u|>t\})\left(pt^{p-1}\log t+t^{p-1}\right)dt \;=\;\int_1^\infty \left(\int_{\mathbb{H}^N}\chi_{\{|u|>t\}}(x)dV_{\mathbb{H}}\right)\left(pt^{p-1}\log t+t^{p-1}\right)dt\\
    &=\int_{\mathbb{H}^N} \left(\int_1^\infty \chi_{\{|u|>t\}}(x)\left(pt^{p-1}\log t+t^{p-1}\right)dt\right)dV_{\mathbb{H}}\\
    &=\int_{\{|u| \geq 1\}} \left(\int_1^{\infty} \chi_{\{|u|>t\}}(x)\left(pt^{p-1}\log t+t^{p-1}\right)dt\right)dV_{\mathbb{H}}\\
    &=\int_{\{|u| \geq 1\}}\left(\int_1^{|u(x)|} \left(pt^{p-1}\log t+t^{p-1}\right)dt\right)dV_{\mathbb{H}}
    =\int_{\{|u| \geq 1\}} |u|^p \log |u| dV_{\mathbb{H}}.
\end{align*}
From the two above identities, we derive
\begin{align*}
    \int_{\mathbb{H}^N}|u|^p\log |u|dV&=\int_{\{|u| < 1\}}|u|^p\log |u| dV_{\mathbb{H}}+\int_{\{|u| \geq 1\}}|u|^p\log |u| dV_{\mathbb{H}}\\
    &=\int_0^1 V(\{|u|>t\})\left(pt^{p-1}\log t+t^{p-1}\right)dt+\int_1^\infty V(\{|u|>t\})\left(pt^{p-1}\log t+t^{p-1}\right)dt\\
    &=\int_0^\infty V(\{|u|>t\})\left(pt^{p-1}\log t+t^{p-1}\right)dt\\
    &=\int_0^\infty V(\{u_e^\sharp>t\})\left(pt^{p-1}\log t+t^{p-1}\right)dt\\
    &=\int_{\mathbb{R}^N}(u_e^\sharp)^p\log (u_e^\sharp) dx,
\end{align*}
as desired. Using  \cite[Theorem~1.1]{PD03},
\begin{align*}
        \int_{\mathbb{H}^N} |u|^p\log|u|dV&=\int_{\mathbb{R}^N}(u_e^\sharp)^p\log (u_e^\sharp) dx
        \leq \dfrac{N}{p^2} \ln \left[\mathcal{L}_{N,p} \int_{\mathbb{R}^N}|\nabla (u_e^\sharp)|^pdx\right]\\
        &\leq \dfrac{N}{p^2}\ln\left[\mathcal{L}_{N,p}\left(\int_{\mathbb{H}^N}|\nabla_{\mathbb{H}} u|^pdV_{\mathbb{H}}-\lambda(N,p)\int_{\mathbb{H}^N}|u|^pdV_{\mathbb{H}}\right)\right].
    \end{align*}
   This completes the proof. 
\end{proof}
\medskip


\section{Gaussian Poincar\'e and Gaussian logarithmic Sobolev inequalities on $\mathbb{H}^N$-Proof of Theorem \ref{T2}}
We now present the proof of Theorem~\ref{T2} by using Theorem \ref{T1.1}, establishing the Gaussian Poincaré–type logarithmic Sobolev inequalities in the hyperbolic setting. Recall that
$$dm=G^{-1}e^{-\frac{\rho^2}{2}}dV_{\mathbb{H}}:=\rho_1 dV_{\mathbb{H}},$$
where $G:=\int_{\mathbb{H}^N}e^{-\rho^2/2}dV_{\mathbb{H}}$, and $\rho=\rho(x)=d_{\mathbb{H}}(x, 0)=\ln\left(\dfrac{1+|x|}{1-|x|}\right).$
\begin{proof}[Proof of Theorem \ref{T2}]
    First, assume that $\int_{\mathbb{H}^N}u^2dm=1$, then $\int_{\mathbb{H}^N}u^2\rho_1 dV_{\mathbb{H}}=1$. Set $v:=u\sqrt{\rho_1}$, thus $$\int_{\mathbb{H}^N}v^2dV_{H}=1.$$
By applying Theorem \ref{T1.1}, we have
\begin{equation}\label{eq 4.1}
    \int_{\mathbb{H}^N} |v|^2\log(|v|^2)dV_{\mathbb{H}} \leq \dfrac{N}{2}\log\left[\mathcal{L}_2\left(\int_{\mathbb{H}^N}|\nabla_{\mathbb{H}}v|^2dV_{\mathbb{H}}-\dfrac{N^2(N-1)}{4(N+2)}\int_{\mathbb{H}^N}|v|^2dV_{\mathbb{H}}\right)\right].
\end{equation}
By computing directly, 
\begin{align*}
    \text{LHS}\eqref{eq 4.1}&=\int_{\mathbb{H}^N} u^2 \log(u^2)dm+\int_{\mathbb{H}^N} u^2 \log(\rho_1)dm\\
    &=\int_{\mathbb{H}^N} u^2 \log(u^2)dm-\log(G)-\dfrac{1}{2}\int_{\mathbb{H}^N} u^2 \rho^2 dm,
\end{align*}
\begin{align*}
    \text{RHS}\eqref{eq 4.1}&=\dfrac{N}{2}\log\left(\dfrac{\mathcal{L}_2}{\alpha}\right)+\dfrac{N}{2}\log\left[e \dfrac{\alpha}{e} \left(\int_{\mathbb{H}^N}|\nabla_{\mathbb{H}}v|^2dV_{\mathbb{H}}-\dfrac{N^2(N-1)}{4(N+2)}\int_{\mathbb{H}^N}|v|^2dV_{\mathbb{H}}\right)\right]\\
    &\leq \dfrac{N}{2}\log\left(\dfrac{\mathcal{L}_2}{\alpha}\right)+\dfrac{N}{2} \dfrac{\alpha}{e} \left(\int_{\mathbb{H}^N}|\nabla_{\mathbb{H}}v|^2dV_{\mathbb{H}}-\dfrac{N^2(N-1)}{4(N+2)}\int_{\mathbb{H}^N}|v|^2dV_{\mathbb{H}}\right).
\end{align*}
Hence, from \eqref{eq 4.1}, it follows us
\begin{align*}
    \int_{\mathbb{H}^N} u^2 \log(u^2)dm &\leq \dfrac{N}{2} \dfrac{\alpha}{e} \left(\int_{\mathbb{H}^N}|\nabla_{\mathbb{H}}v|^2dV_{\mathbb{H}}-\dfrac{N^2(N-1)}{4(N+2)}\int_{\mathbb{H}^N}|v|^2dV_{\mathbb{H}}\right)+\dfrac{1}{2}\int_{\mathbb{H}^N} u^2 \rho^2 dm\\
    & + \log(G)+\dfrac{N}{2}\log\left(\dfrac{\mathcal{L}_2}{\alpha}\right).
\end{align*}
Next, we will estimate $G$. Indeed, with $\omega_{N-1}$ the area of the unit sphere $\mathbb{S}^{N-1}$,
\begin{align*}
    G=\int_{\mathbb{H}^N}e^{-\rho^2/2}dV_{\mathbb{H}}&=\omega_{N-1} \int_0^\infty e^{-r^2/2}\sinh^{N-1}rdr\leq \dfrac{\omega_{N-1}}{2^{N-1}}\int_0^\infty e^{-\frac{r^2}{2}+(N-1)r}dr\\
    &=e^{\frac{(N-1)^2}{2}}\dfrac{\omega_{N-1}}{2^{N-1}} \int_0^\infty e^{-\frac{1}{2}(r-(N-1))^2}dr\\
    &\leq e^{\frac{(N-1)^2}{2}}\dfrac{\omega_{N-1}}{2^{N-1}} \int_{-\infty}^\infty e^{-\frac{1}{2}t^2}dt=\dfrac{\sqrt{2 \pi}}{2^{N-1}}e^{\frac{(N-1)^2}{2}}\omega_{N-1}.
\end{align*}
Therefore,
\begin{align*}
    \log(G)+\dfrac{N}{2}\log\left(\dfrac{\mathcal{L}_2}{\alpha}\right) &\leq \log\left[\alpha^{-N/2}(\mathcal{L}_2)^{N/2}\dfrac{\sqrt{2 \pi}}{2^{N-1}}e^{\frac{(N-1)^2}{2}}\omega_{N-1}\right]\\
    &=\log\left[C_1\alpha^{-N/2}\right],
\end{align*}
where $C_1:= (\mathcal{L}_2)^{N/2}\dfrac{\sqrt{2 \pi}}{2^{N-1}}e^{\frac{(N-1)^2}{2}}\omega_{N-1}.$
Furthermore,
\begin{align*}
    |\nabla_{\mathbb{H}} v|^2=|\nabla_{\mathbb{H}} (u \sqrt{\rho_1})|^2=\rho_1|\nabla_{\mathbb{H}} u|^2-\nabla_{\mathbb{H}}\left(\dfrac{u^2}{2}\rho_1\right)\nabla_{\mathbb{H}}\left(\dfrac{\rho^2}{2}\right)-\dfrac{1}{4}u^2 \rho^2 \rho_1,
\end{align*}
and
\begin{align*}
    -\int_{\mathbb{H}^N}\nabla_{\mathbb{H}}\left(\dfrac{u^2}{2}\rho_1\right)\nabla_{\mathbb{H}}\left(\dfrac{\rho^2}{2}\right) dV_{\mathbb{H}}&=\int_{\mathbb{H}^N} \dfrac{u^2}{2}\rho_1 \Delta_{\mathbb{H}} \left(\dfrac{\rho^2}{2}\right)dV_{\mathbb{H}}\\
    &=\int_{\mathbb{H}^N} \dfrac{u^2}{2}\rho_1 \left(1+(N-1)\rho\coth \rho\right)dV_{\mathbb{H}}.
\end{align*}
Thus,
\begin{align*}
    \int_{\mathbb{H}^N} |\nabla_{\mathbb{H}}v|^2 dV_{\mathbb{H}}=\int_{\mathbb{H}^N} |\nabla_{\mathbb{H}}u|^2 dm-\dfrac{1}{4} \int_{\mathbb{H}^N} u^2\rho^2 dm+\dfrac{1}{2}\int_{\mathbb{H}^N} u^2 \left(1+(N-1)\rho\coth \rho\right)dm,
\end{align*}
and
\begin{align*}
    \int_{\mathbb{H}^N} u^2 \log(u^2)dm &\leq \dfrac{N}{2} \dfrac{\alpha}{e} \left(\int_{\mathbb{H}^N} |\nabla_{\mathbb{H}}u|^2 dm-\dfrac{1}{4} \int_{\mathbb{H}^N} u^2\rho^2 dm+\dfrac{1}{2}\int_{\mathbb{H}^N} u^2 \left(1+(N-1)\rho\coth \rho\right)dm\right)\\
    &-\dfrac{\alpha N^3(N-1)}{8e(N+2)}\int_{\mathbb{H}^N}|v|^2dV_{\mathbb{H}}+\dfrac{1}{2}\int_{\mathbb{H}^N} u^2 \rho^2 dm+ \log\left[\alpha^{-N/2}C_1\right]\\
    &=\dfrac{\alpha N}{2e}\int_{\mathbb{H}^N} |\nabla_{\mathbb{H}}u|^2 dm+\left(\dfrac{1}{2}-\dfrac{\alpha N}{8e}\right)\int_{\mathbb{H}^N} u^2 \rho^2 dm\\
    &+\dfrac{\alpha N}{4e}\int_{\mathbb{H}^N} u^2 \left(1+(N-1)\rho\coth \rho\right)dm\\
    &-\dfrac{\alpha N^3(N-1)}{8e(N+2)}\int_{\mathbb{H}^N}u^2dm+ \log\left[\alpha^{-N/2}C_1\right].
\end{align*}
By choosing $\alpha=4e/N$, we obtain
\begin{align}\label{GaussLogSob}
    \int_{\mathbb{H}^N} u^2 \log(u^2)dm &\leq 2\int_{\mathbb{H}^N} |\nabla_{\mathbb{H}}u|^2 dm+\int_{\mathbb{H}^N} u^2 \left(1+(N-1)\rho\coth \rho\right)dm \nonumber\\
    &-\dfrac{N^2(N-1)}{2(N+2)}\int_{\mathbb{H}^N}u^2dm+\log\left(C_1\left(\dfrac{4e}{N}\right)^{-N/2}\right) \nonumber\\
    &=2\int_{\mathbb{H}^N} |\nabla_{\mathbb{H}}u|^2 dm \nonumber\\
    &+\int_{\mathbb{H}^N} \left[(N-1)\left(\rho\coth \rho-1\right)-\dfrac{N^2(N-1)}{2(N+2)}+\log\left(C_1\left(\dfrac{Ne}{4}\right)^{N/2}\right)\right]u^2dm.
\end{align}
In general, without the condition $\|u\|_{L^2(\mathbb{H}^N,dm)}=1$,  by a standard scaling argument where $u$ is replaced by $\dfrac{u}{\|u\|_{L^2(\mathbb{H}^N,dm)}}$, we can derive 
\begin{align*}
    &\int_{\mathbb{H}^N} u^2 \log(u^2)dm -\log\left(\int_{\mathbb{H}^N} u^2 dm\right)\int_{\mathbb{H}^N} u^2 dm \nonumber\\
    &\leq 2\int_{\mathbb{H}^N} |\nabla_{\mathbb{H}}u|^2 dm+\int_{\mathbb{H}^N} \left[(N-1)\left(\rho\coth \rho-1\right)-\dfrac{N^2(N-1)}{2(N+2)}+\log\left(C_2\right)\right]u^2dm,
\end{align*}
where
\begin{equation*}
    C_2:=(\mathcal{L}_2)^{N/2}\dfrac{\sqrt{2 \pi}}{2^{N-1}}e^{\frac{(N-1)^2}{2}}\omega_{N-1} \left(\dfrac{Ne}{4}\right)^{N/2},
\end{equation*}
as desired.
\end{proof}

\begin{remark}
{\rm
We first observe that, as $N \to \infty$, the region in which the inequality
\[
(N-1)\bigl(\rho \coth \rho - 1\bigr) - \frac{N^{2}(N-1)}{2(N+2)} < 0
\]
holds becomes increasingly large. In particular, for sufficiently large dimensions, the Gaussian logarithmic Sobolev inequality \eqref{GaussLogSob} provides a relatively accurate and effective estimate.
}
\end{remark}

\begin{remark}
{\rm
We now proceed to compute explicitly the normalization constant
\[
G := \int_{\mathbb{H}^{N}} e^{-\rho^{2}/2}\, dV_{\mathbb{H}}.
\]
Using polar coordinates on the hyperbolic space, we write
\begin{align*}
G
&= \omega_{N-1} \int_{0}^{\infty} e^{-r^{2}/2} \sinh^{N-1} r \, dr \\
&= \omega_{N-1} \int_{0}^{\infty} e^{-r^{2}/2}
\left( \frac{e^{r} - e^{-r}}{2} \right)^{N-1} dr \\
&= \frac{\omega_{N-1}}{2^{N-1}}
\int_{0}^{\infty} e^{-r^{2}/2 + (N-1)r}
\left( 1 - e^{-2r} \right)^{N-1} dr.
\end{align*}
Expanding the last factor by the binomial formula yields
\begin{align*}
G
&= \frac{\omega_{N-1}}{2^{N-1}}
\sum_{k=0}^{N-1} \binom{N-1}{k} (-1)^k
\int_{0}^{\infty} e^{-r^{2}/2 + (N-2k-1)r} \, dr.
\end{align*}
Setting $\alpha_k := N-1 - 2k$ for $k = 0, \dots, N-1$ and completing the square in the exponent, we obtain
\begin{align*}
G
&= \frac{\omega_{N-1}}{2^{N-1}}
\sum_{k=0}^{N-1} \binom{N-1}{k} (-1)^k
e^{\alpha_k^{2}/2}
\int_{-\alpha_k}^{\infty} e^{-r^{2}/2} \, dr \\
&= \frac{\omega_{N-1}}{2^{N-1}}
\sum_{k=0}^{N-1} \binom{N-1}{k} (-1)^k
e^{\alpha_k^{2}/2} \sqrt{2}
\int_{-\alpha_k/\sqrt{2}}^{\infty} e^{-r^{2}} \, dr.
\end{align*}
Recalling that
\[
\int_{-\alpha_k/\sqrt{2}}^{\infty} e^{-r^{2}} \, dr
= \frac{\sqrt{\pi}}{2}
\left( 1 + \operatorname{erf}\!\left( \frac{\alpha_k}{\sqrt{2}} \right) \right),
\]
we finally arrive at the explicit representation
\[
G
= \sqrt{\frac{\pi}{2}}\,
\frac{\omega_{N-1}}{2^{N-1}}
\sum_{k=0}^{N-1} \binom{N-1}{k} (-1)^k
e^{\alpha_k^{2}/2}
\left( 1 + \operatorname{erf}\!\left( \frac{\alpha_k}{\sqrt{2}} \right) \right).
\]

Here, $\operatorname{erf}$ denotes the error function defined by
\[
\operatorname{erf}(x)
:= \frac{2}{\sqrt{\pi}} \int_{0}^{x} e^{-t^{2}} \, dt,
\qquad x \in \mathbb{R},
\]
which is an odd, smooth, and strictly increasing function satisfying
\[
\lim_{x \to \infty} \operatorname{erf}(x) = 1,
\qquad
\lim_{x \to -\infty} \operatorname{erf}(x) = -1.
\]

}
\end{remark}

Next, we prove a generalized Poincaré inequality on $\mathbb{H}^N$ with respect to Gaussian-type measures. The result is motivated by Beckner’s inequality in \cite{B89} and can be viewed as its hyperbolic counterpart. More precisely, we show that the following inequality holds.

\begin{theorem}
Let $0<p\leq 2$. Then, for every $u\in C_0^\infty(\mathbb{H}^N)$, the following inequality holds:
\begin{align}\label{gnrGaussPcr}
    \int_{\mathbb{H}^N} u^2\,dm
    -\left(\int_{\mathbb{H}^N} |u|^p\,dm\right)^{\frac{2}{p}}
    &\leq \frac{2(2-p)}{p}\int_{\mathbb{H}^N} |\nabla_{\mathbb{H}}u|^2\,dm \nonumber\\[0.3em]
    &+\frac{2-p}{p}\int_{\mathbb{H}^N}
    \Bigg[
        (N-1)\big(\rho\coth\rho-1\big)
        -\frac{N^2(N-1)}{2(N+2)}
        +\log(C_2)
    \Bigg] u^2\,dm .
\end{align}
\end{theorem}

\begin{remark}
{\rm
As a particular case, when $p=1$ the inequality reduces to
\begin{align}\label{GaussPcr}
\int_{\mathbb{H}^N} u^2 dm
-\left(\int_{\mathbb{H}^N} |u|
dm\right)^2
&\leq 2\int_{\mathbb{H}^N} |\nabla_{\mathbb{H}} u|^2dm \nonumber\\
& +\int_{\mathbb{H}^N}
\Bigg[(N-1)\big(\rho\coth\rho-1\big)
-\frac{N^2(N-1)}{2(N+2)}
+\log(C_2)\Bigg] u^2 dm.
\end{align}
Here, $C_2$ is defined as \eqref{C2}.
}
\end{remark}

\begin{proof}
    For $q \geq 1$, let us consider 
    $$J(q):=q \log\left(\int_{\mathbb{H}^N} |u|^{2/q}dm\right)=q\log (I(q)),$$
    where $I(q)=\int_{\mathbb{H}^N} |u|^{2/q}dm.$ We have
    $$I'(q)=\int_{\mathbb{H}^N} \log (|u|) |u|^{2/q}\left(-\dfrac{2}{q^2}\right)dm,$$
    and
    $$I''(q)=\int_{\mathbb{H}^N} (\log |u|)^2 |u|^{2/q}\dfrac{4}{q^4}dm+\int_{\mathbb{H}^N}\log(|u|)|u|^{2/q}\dfrac{4}{q^3}dm.$$
    Then,
    $$J'(q)=\log(I(q))+q\dfrac{I'(q)}{I(q)},$$
    and 
    \begin{align*}
        J''(q)&=\dfrac{I'(q)}{I(q)}+\dfrac{(I'(q)+qI''(q))I(q)-q(I'(q))^2}{(I(q))^2}\\
        &=\dfrac{q}{(I(q))^2}\left[I(q)\left(I''(q)+\frac{2}{q}I'(q)\right)-(I'(q))^2\right]\\
        &=\dfrac{q}{(I(q))^2}\left[\left(\int_{\mathbb{H}^N}|u|^{2/q}dm\right)\left(\int_{\mathbb{H}^N}(\log |u|)^2|u|^{2/q}\dfrac{4}{q^4}dm\right)-\left(\int_{\mathbb{H}^N} \log (|u|) |u|^{2/q}\left(-\dfrac{2}{q^2}\right)dm\right)^2\right]\\
        & \geq 0,
    \end{align*}
    which follows from the Cauchy-Schwarz inequality. Hence, $J(q)$ is convex and so is $e^{J(q)}$.
    Moreover, the function
    $$K(q):=\dfrac{e^{J(1)}-e^{J(q)}}{q-1}$$
    is monotonically non-increasing for $q > 1$. Indeed, we compute
    $$K'(q)=\dfrac{J'(q)e^{J(q)}+e^{J(q)}-qe^{J(q)}J'(q)-e^{J(1)}}{(q-1)^2}.$$
   Define
\[
F(q) := J'(q)e^{J(q)} + e^{J(q)} - q e^{J(q)}J'(q) - e^{J(1)} .
\]
A direct computation gives
\[
F'(q)
= (1-q)J''(q)e^{J(q)} + (1-q)\big(J'(q)\big)^2 e^{J(q)} \leq 0 ,
\]
for all $q \geq 1$. Therefore, $F$ is non-increasing on $[1,\infty)$, and consequently
\[
F(q) \leq F(1)=0 \qquad \text{for all } q \geq 1 .
\]
 Hence, $K'(q) \leq 0$ for all $q >1$, which allows us
    $$K(q) \leq \lim_{q \rightarrow 1} K(q)$$
    for all $q >1$. Moreover,
    $$\lim_{q \rightarrow 1} K(q)=\lim_{q \rightarrow 1} \dfrac{e^{J(1)}-e^{J(q)}}{q-1}=-\lim_{q \rightarrow 1} \left(e^{J(q)}\right)'.$$
    Since
    \begin{align*}
        \left(e^{J(q)}\right)'=\left(\int_{\mathbb{H}^N}|u|^{2/q}dm\right)^q \log\left(\int_{\mathbb{H}^N}|u|^{2/q}dm\right)-\dfrac{2}{q}\left(\int_{\mathbb{H}^N}|u|^{2/q}dm\right)^{q-1}\int_{\mathbb{H}^N} |u|^{2/q}\log|u|dm,
    \end{align*}
    we get
    $$\lim_{q \rightarrow 1} K(q)=-\lim_{q \rightarrow 1} \left(e^{J(q)}\right)'=2\int_{\mathbb{H}^N} |u|^{2}\log|u|dm-\left(\int_{\mathbb{H}^N}|u|^{2}dm\right) \log\left(\int_{\mathbb{H}^N}|u|^{2}dm\right).$$
    Thus, it implies that 
    \begin{align*}
        &\int_{\mathbb{H}^N}|u|^2dm-\left(\int_{\mathbb{H}^N}|u|^{2/q}dm\right)^q\\
        &
        \leq (q-1)\left[\int_{\mathbb{H}^N} |u|^{2}\log(|u|^2)dm-\left(\int_{\mathbb{H}^N}|u|^{2}dm\right) \log\left(\int_{\mathbb{H}^N}|u|^{2}dm\right)\right]\\
        &\leq 2(q-1)\int_{\mathbb{H}^N} |\nabla_{\mathbb{H}}u|^2 dm
        +(q-1)\int_{\mathbb{H}^N} \left[(N-1)\left(\rho\coth \rho-1\right)-\dfrac{N^2(N-1)}{2(N+2)}+\log(C_2)\right]u^2dm,
    \end{align*}
    for all $q \geq 1$. Here, we used \eqref{gnrGaussLogSob}. This is exactly our assertion when we choose $p=\frac{2}{q}$.
\end{proof}

We also have the following results: 
\begin{theorem}
   Let $q_0>0$ be fixed, and let $a>0$ and $b\in\mathbb{R}$ be such that
\[
a q_0 + b = 1 .
\]
Then, for every $q \geq q_0$ and every $u\in C_0^\infty(\mathbb{H}^N)$, the following inequality holds:
    \begin{align*}
    &\int_{\mathbb{H}^N}|u|^{2}dm-\left(\int_{\mathbb{H}^N}|u|^{2/(aq+b)}dm\right)^{aq+b}
        \leq 2a(q-q_0)\int_{\mathbb{H}^N} |\nabla_{\mathbb{H}}u|^2 dm\\
       & +a(q-q_0)\int_{\mathbb{H}^N} \left[(N-1)\left(\rho\coth \rho-1\right)-\dfrac{N^2(N-1)}{2(N+2)}+\log(C_2)\right]u^2dm.
\end{align*}
\end{theorem}
\begin{proof}
    The proof proceeds along the same lines as that of the previous theorem. Given a positive $q_0$, for $q \geq q_0>0$,  let us consider 
    $$J(q):=f(q) \log\left(\int_{\mathbb{H}^N} |u|^{g(q)}dm\right)=f(q)\log (I(q)),$$
    where $I(q)=\int_{\mathbb{H}^N} |u|^{g(q)}dm$, and $f(q)$, $g(q)$ are two functions which will be determined later. We have
    $$I'(q)=\int_{\mathbb{H}^N} \log (|u|) |u|^{g(q)}g'(q)dm,$$
    and
    $$I''(q)=\int_{\mathbb{H}^N} (\log |u|)^2 |u|^{g(q)}(g'(q))^2dm+\int_{\mathbb{H}^N}\log(|u|)|u|^{g(q)}g''(q)dm.$$
    Then,
    $$J'(q)=f'(q)\log(I(q))+f(q)\dfrac{I'(q)}{I(q)},$$
    and 
    \begin{align*}
        &J''(q) =f''(q)\log(I(q))+f'(q)\dfrac{I'(q)}{I(q)}+f'(q)\dfrac{I'(q)}{I(q)}+f(q)\dfrac{I''(q)I(q)-(I'(q))^2}{(I(q))^2}\\
        &=f''(q)\log(I(q))+\dfrac{f(q)}{(I(q))^2}\left[I(q)\left(I''(q)+\frac{2f'(q)}{f(q)}I'(q)\right)-(I'(q))^2\right]\\
        &=f''(q)\log(I(q))\\
        &+\dfrac{f(q)}{(I(q))^2}\\
        &\times\left[\left(\int_{\mathbb{H}^N}|u|^{g(q)}dm\right)\left(\left(\dfrac{2f'(q)}{f(q)}g'(q)+g''(q)\right)\int_{\mathbb{H}^N}\log(|u|)|u|^{g(q)}dm+(g'(q))^2\int_{\mathbb{H}^N}(\log |u|)^2|u|^{g(q)}dm\right)\right.\\
        &\left.-(g'(q))^2\left(\int_{\mathbb{H}^N} \log (|u|) |u|^{g(q)}dm\right)^2\right]
         \geq 0,
    \end{align*}
    which follows from Cauchy-Schwarz inequality, provided that
    \begin{equation}\label{cdt1}
        \left\{\begin{matrix}
f''(q)=0 \\
\dfrac{2f'(q)}{f(q)}g'(q)+g''(q)=0 \\
f(q) \geq 0
\end{matrix}\right.
    \end{equation}
for all $q \geq q_0$. Hence, $J(q)$ is convex and so is $e^{J(q)}$.
Besides, the function
    $$K(q):=\dfrac{e^{J(q_0)}-e^{J(q)}}{q-q_0}$$
    is monotonically non-increasing for $q > q_0$, which implies that
    $$K(q) \leq \lim_{q \rightarrow q_0} K(q)$$
    for all $q >q_0$. Moreover,
    $$\lim_{q \rightarrow q_0} K(q)=\lim_{q \rightarrow q_0} \dfrac{e^{J(q_0)}-e^{J(q)}}{q-q_0}=-\lim_{q \rightarrow q_0} \left(e^{J(q)}\right)'.$$
    Since
    \begin{align*}
        \left(e^{J(q)}\right)'&=f'(q)\left(\int_{\mathbb{H}^N}|u|^{g(q)}dm\right)^{f(q)} \log\left(\int_{\mathbb{H}^N}|u|^{g(q)}dm\right)\\
        &+f(q)g'(q)\left(\int_{\mathbb{H}^N}|u|^{g(q)}dm\right)^{f(q)-1}\int_{\mathbb{H}^N} |u|^{g(q)}\log|u|dm,
    \end{align*}
    we get
    \begin{align*}
        \lim_{q \rightarrow q_0} K(q)=-\lim_{q \rightarrow q_0} \left(e^{J(q)}\right)'
        &=-f'(q_0)\left(\int_{\mathbb{H}^N}|u|^{g(q_0)}dm\right)^{f(q_0)} \log\left(\int_{\mathbb{H}^N}|u|^{g(q_0)}dm\right)\\
        &-f(q_0)g'(q_0)\left(\int_{\mathbb{H}^N}|u|^{g(q_0)}dm\right)^{f(q_0)-1}\int_{\mathbb{H}^N} |u|^{g(q_0)}\log|u|dm.
    \end{align*}
    Thus,
    \begin{align*}
        &\left(\int_{\mathbb{H}^N}|u|^{g(p_0)}dm\right)^{f(q_0)}-\left(\int_{\mathbb{H}^N}|u|^{g(q)}dm\right)^{f(q)}\\
        &\leq (q-q_0)\left[-f'(q_0)\left(\int_{\mathbb{H}^N}|u|^{g(q_0)}dm\right)^{f(q_0)} \log\left(\int_{\mathbb{H}^N}|u|^{g(q_0)}dm\right)\right.\\
        &\left.-f(q_0)g'(q_0)\left(\int_{\mathbb{H}^N}|u|^{g(q_0)}dm\right)^{f(q_0)-1}\int_{\mathbb{H}^N} |u|^{g(q_0)}\log|u|dm\right]\\
        &=(q-q_0)\left[-f(q_0)\dfrac{g'(q_0)}{g(q_0)}\left(\int_{\mathbb{H}^N}|u|^{g(q_0)}dm\right)^{f(q_0)-1}\int_{\mathbb{H}^N} |u|^{g(q_0)}\log\left(|u|^{g(q_0)}\right)dm\right.\\
        &\left.-f'(q_0)\left(\int_{\mathbb{H}^N}|u|^{g(q_0)}dm\right)^{f(q_0)} \log\left(\int_{\mathbb{H}^N}|u|^{g(q_0)}dm\right)\right]
    \end{align*}
    for all $q \geq q_0$. By choosing $q_0$ such that
    \begin{equation}\label{cdt2}
        \left\{\begin{matrix}
f(q_0)=1, g(q_0)=2 \\
-f(q_0)\dfrac{g'(q_0)}{g(q_0)}=f'(q_0)>0
\end{matrix}\right.,
    \end{equation}
we get
\begin{align*}
    &\int_{\mathbb{H}^N}|u|^{2}dm-\left(\int_{\mathbb{H}^N}|u|^{g(q)}dm\right)^{f(q)}\\
        &\leq(q-q_0)f'(q_0)\left[\int_{\mathbb{H}^N} |u|^{2}\log\left(|u|^{2}\right)dm-\left(\int_{\mathbb{H}^N}|u|^{2}dm\right)\log\left(\int_{\mathbb{H}^N}|u|^{2}dm\right)\right]\\
        &\leq 2(q-q_0)f'(q_0)\int_{\mathbb{H}^N} |\nabla_{\mathbb{H}}u|^2 dm\\
        &+(q-q_0)f'(q_0)\int_{\mathbb{H}^N} \left[(N-1)\left(\rho\coth \rho-1\right)-\dfrac{N^2(N-1)}{2(N+2)}+\log(C_2)\right]u^2dm,
\end{align*}
for all $q \geq q_0$.
Next, we will find out the explicit expressions of $f(q)$ and $g(q)$ such that they satisfy \eqref{cdt1} and \eqref{cdt2}. From \eqref{cdt1}, we can choose
$$f(q)=aq+b\;\text{and}\;g(q)=\dfrac{d}{a(aq+b)}+c,$$
where $a, b, c, d \in \mathbb{R}$, and $aq+b>0$ for $q\geq q_0>0$ with some $q_0$. Substituting these formulas into \eqref{cdt2}, we have
$$\left\{\begin{matrix}
aq_0+b=1\\
\dfrac{d}{a(aq_0+b)}+c=2 \\
-g'(q_0)=\dfrac{d}{(aq_0+b)^2}=2a>0
\end{matrix}\right.,$$
which is equivalent to
$$\left\{\begin{matrix}
aq_0+b=1\\
c=0 \\
d=2a>0
\end{matrix}\right..$$
Therefore, we can determine
$$f(q)=aq+b\;\text{and}\;g(q)=\dfrac{2}{aq+b}$$
such that $a>0$ and $aq_0+b=1$, which suffice to guarantee that $aq+b\geq aq_0+b=1>0$ for all $q \geq q_0>0$, for some $q_0>0$.
Hence, we derive the following Poincar\'{e} inequality on $\mathbb{H}^N$ as follows:
\begin{align*}
    &\int_{\mathbb{H}^N}|u|^{2}dm-\left(\int_{\mathbb{H}^N}|u|^{2/(aq+b)}dm\right)^{aq+b}\\
        &\leq 2a(q-q_0)\int_{\mathbb{H}^N} |\nabla_{\mathbb{H}}u|^2 dm\\
        &+a(q-q_0)\int_{\mathbb{H}^N} \left[(N-1)\left(\rho\coth \rho-1\right)-\dfrac{N^2(N-1)}{2(N+2)}+\log(C_2)\right]u^2dm,
\end{align*}
for all $q \geq q_0$, as desired.
\end{proof}
\begin{corollary}
    Given $\lambda>0$ arbitrarily, we have
    \begin{align*}
    &\int_{\mathbb{H}^N}|u|^{2}dm-\left(\int_{\mathbb{H}^N}|u|^{\frac{4}{\lambda+2}}dm\right)^{\frac{\lambda+2}{2}}\\
        &\leq \lambda\int_{\mathbb{H}^N} |\nabla_{\mathbb{H}}u|^2 dm\\
        &+\dfrac{\lambda}{2}\int_{\mathbb{H}^N} \left[(N-1)\left(\rho\coth \rho-1\right)-\dfrac{N^2(N-1)}{2(N+2)}+\log(C_2)\right]u^2dm,
\end{align*}
\end{corollary}
\begin{proof}
    It is directly from the above theorem. We choose $q$ such that $2a(q-q_0)=\lambda>0$, then $q=\frac{\lambda}{2a}+q_0 \geq q_0$,
    $$\dfrac{2}{aq+b}=\dfrac{2}{\lambda/2+aq_0+b}=\dfrac{2}{\lambda/2+1}=\dfrac{4}{\lambda+2},$$
    and 
    $aq+b=2\dfrac{\lambda+2}{4}=\dfrac{\lambda+2}{2}.$
\end{proof}

\medskip


\section{Extended Beckner inequalities on the hyperbolic space $\mathbb{H}^N$}

It is well known that the classical Beckner inequality can be interpreted as an interpolation between the Poincaré inequality and Gross’s logarithmic Sobolev inequality. Motivated by this perspective, we now turn to the study of an extended Beckner inequality on the hyperbolic space $\mathbb{H}^N$, incorporating a modified measure in comparison with the results obtained in the previous theorems.

The core idea of our approach is to exploit the interplay between the heat kernel on $\mathbb{H}^N$ and the fundamental properties of the associated heat semigroup. This framework allows us to capture the interpolating nature of the inequality in the hyperbolic setting and to extend the classical Euclidean theory to $\mathbb{H}^N$

We begin by recalling the explicit expression of  the \textbf{heat kernel $p_N(\rho,t)$} on the hyperbolic space $\mathbb{H}^N$ which is given by:
 $$p_N(\rho, t)=\dfrac{(-1)^m}{2^m \pi ^m}\dfrac{1}{(4\pi t)^{1/2}}\left(\dfrac{1}{\sinh \rho}\dfrac{\partial}{\partial \rho}\right)^m e^{-m^2t-\frac{\rho^2}{4t}},$$ 
 if $N=2m+1$;
 $$p_{N}(\rho, t)=\dfrac{(-1)^m}{2^{m+5/2}\pi^{m+3/2}}t^{-\frac{3}{2}}e^{-\frac{(2m+1)^2}{4}t}\left(\dfrac{1}{\sinh \rho} \dfrac{\partial}{\partial \rho}\right)^m \displaystyle\int_{\rho}^\infty \dfrac{se^{-\frac{s^2}{4t}}}{(\cosh s-\cosh \rho)^{1/2}}ds,$$ 
 if $N=2m+2$. It means that, with $\rho:=d_{\mathbb{H}^N}(x,y)$, the solution of the Cauchy problem for the heat equation on $\mathbb{H}^N$
 $$ \left\{\begin{matrix}
\dfrac{\partial v}{\partial t}&=\Delta_{\mathbb{H}}v \\
v|_{t=0}&=v_0(x)
\end{matrix}\right.$$
is given by
$$v(x,t)=\int_{\mathbb{H}^N}p_N(\rho, t)v_0(y)dV_{\mathbb{H}}.$$ 
See \cite{EBD1, EBD2, GM98} for details. We define the heat semigroup $\{P_s: s>0\}$ on $\mathbb{H}^N$ as follows:
$$P_sf(x)=\int_{\mathbb{H}^N}f(y)p_N(\rho, \alpha s)dV_{\mathbb{H}},$$
with  $\alpha>0$.
As a direct consequence of the heat kernel (see \cite{EBD1, GP20}), $\{P_s: s>0\}$ has the following elementary properties:
\begin{lemma}
    Suppose $f: \mathbb{H}^N \rightarrow \mathbb{R}$ is bounded and continuous. The followings hold:
    \begin{itemize}
        \item [i)] $P_sf$ solves the heat equation on $\mathbb{H}^N$, i.e. $\partial_s P_sf=\alpha\Delta_{\mathbb{H}}P_sf$;
        \item [ii)] $P_sf \rightarrow f$ as $s \rightarrow 0$, and $P_1f(0)=\displaystyle\int_{\mathbb{H}^N}f(y)p_N(\rho(y),\alpha)dV_{\mathbb{H}}$, where $\rho(y):=d_{\mathbb{H}^N}(y,0)$;
        \item [iii)] For all $s, t \geq 0$, $P_s \circ P_t=P_{s+t}$;
        \item [iv)] If $\nabla_{\mathbb{H}}f$ is bounded and continuous, $\nabla_{\mathbb{H}}P_sf=P_s \nabla_{\mathbb{H}}f$. Moreover, $\Delta_{\mathbb{H}}P_sf=P_s \Delta_{\mathbb{H}}f$.
    \end{itemize}
\end{lemma}
We define the measure
\[
d\gamma := p_N\big(\rho(y),\alpha\big)\, dV_{\mathbb{H}} .
\]
With respect to this measure, we obtain the following extended Beckner inequality.

\begin{theorem}
Let $q\geq 2$ and $1\leq p\leq q$. Then, for every sufficiently smooth function
$f\colon \mathbb{H}^N \to \mathbb{R}$, the following inequality holds:
\[
\left(\int_{\mathbb{H}^N} |f|^q \, d\gamma\right)^{\!2/q}
-
\left(\int_{\mathbb{H}^N} |f|^p \, d\gamma\right)^{\!2/p}
\leq
2\alpha (q-p)
\left(\int_{\mathbb{H}^N} |\nabla_{\mathbb{H}} f|^q \, d\gamma\right)^{\!2/q}.
\]
\end{theorem}
\begin{proof}
    The proof follows \cite{G13}. More clearly, without loss of generality, we assume that $f$ is smooth, positive, and bounded. For $0 \leq s \leq 1$, we consider
    $$\phi_s(x)=[P_s[(P_{1-s}f^p)^{q/p}](x)]^{2/q}.$$
    Then, by the property $ii)$ of $\{P_s\}$, $\phi_0(x)=[(P_1 f^p)(x)]^{2/p}$ and $\phi_1(x)=[P_1(f^q)(x)]^{2/q}$. Hence,
    \begin{equation}\label{extBecknerineq}
        \left(\int_{\mathbb{H}^N}f^q d\gamma\right)^{2/q}-\left(\int_{\mathbb{H}^N}f^p d\gamma\right)^{2/p}=\phi_1(0)-\phi_0(0)=\int_0^1 \partial_s \phi_s(0)ds.
    \end{equation}
    Next, we will compute the derivative of $\phi_s(x)$ with respect to $s$. Indeed, setting $g_s:=P_{1-s}(f^p)$ and $a_s=[P_s(g_s^{q/p})]^{2/q-1}$, we have
    \begin{align*}
        \partial_s \phi_s=\frac{2}{q}a_s \partial_s[P_s(g_s^{q/p})]=\frac{2}{q}a_s (\partial_s P_s)(g_s^{q/p})+\frac{2}{p}a_s P_s\left(g_s^{q/p-1}\partial_s g_s\right).
    \end{align*}
   Since $\partial_s P_s=\alpha \Delta_{\mathbb{H}}P_s=\alpha P_s\Delta_{\mathbb{H}}$,
    \begin{align*}
        \frac{2}{q}a_s (\partial_s P_s)(g_s^{q/p})=\dfrac{2\alpha}{q} a_s \Delta_{\mathbb{H}}P_s(g_s^{q/p})=\dfrac{2\alpha}{q}a_sP_s\Delta_{\mathbb{H}}(g_s^{q/p}).
    \end{align*}
    Moreover, from
    $$\Delta_{\mathbb{H}}(g_s^{q/p})=\text{div}_{\mathbb{H}}\left(\nabla_{\mathbb{H}}(g_s^{q/p})\right)=\text{div}_{\mathbb{H}} \left(\frac{q}{p}g_s^{q/p-1}\nabla_{\mathbb{H}}g_s\right)=\dfrac{q(q-p)}{p^2}g_s^{q/p-2}|\nabla_{\mathbb{H}}g_s|^2+\dfrac{q}{p} g_s^{q/p-1}\Delta_{\mathbb{H}}g_s,$$
    we get
    \begin{align*}
        \frac{2}{q}a_s (\partial_s P_s)(g_s^{q/p})&=\dfrac{2\alpha(q-p)}{p^2}a_s P_s\left(g_s^{q/p-2}|\nabla_{\mathbb{H}}g_s|^2\right)+\dfrac{2\alpha}{p}a_s P_s\left(g_s^{q/p-1}\Delta_{\mathbb{H}}g_s\right)\\
        &=\dfrac{2\alpha(q-p)}{p^2}a_s P_s\left(g_s^{q/p-2}|\nabla_{\mathbb{H}}g_s|^2\right)-\dfrac{2}{p}a_s P_s\left(g_s^{q/p-1}\partial_s g_s\right),
    \end{align*}
    since $\partial_s P_{1-s}=-\alpha\Delta_{\mathbb{H}}P_{1-s}$, or $\partial_s g_s=-\alpha\Delta_{\mathbb{H}}g_s$.
    Hence,
    $$\partial_s \phi_s=\dfrac{2\alpha(q-p)}{p^2}a_s P_s\left(g_s^{q/p-2}|\nabla_{\mathbb{H}}g_s|^2\right).$$
    Furthermore,
    $$\nabla_{\mathbb{H}}g_s=\nabla_{\mathbb{H}}\left(P_{1-s}f^p\right)=P_{1-s}(\nabla_{\mathbb{H}} f^p)=pP_{1-s}(f^{p-1}\nabla_{\mathbb{H}}f),$$
    which implies that
    $$\partial_s\phi_s=2\alpha(q-p)a_sP_s\left(g_s^{q/p-2}\left|P_{1-s}(f^{p-1}\nabla_{\mathbb{H}}f)\right|^2\right).$$
    By Holder's inequality,
    $$\left|P_{1-s}(f^{p-1}\nabla_{\mathbb{H}}f)\right| \leq P_{1-s}\left(f^{p-1}|\nabla_{\mathbb{H}}f|\right)\leq \left(P_{1-s}f^p\right)^{(p-1)/p}\left(P_{1-s}|\nabla f|^p\right)^{1/p}=g_s^{(p-1)/p}\left(P_{1-s}|\nabla f|^p\right)^{1/p}.$$
    Thus,
    $$\partial_s\phi_s \leq 2\alpha(q-p)a_sP_s(g_s^{q/p-2/p}\left(P_{1-s}|\nabla_{\mathbb{H}}f^p\right)^{2/p}).$$
    Next, we assume $q>2$. With the case $q=2$, we can proceed in a similar way. Indeed, using Holder's inequality again, 
    \begin{align*}
        \partial_s\phi_s &\leq 2\alpha(q-p)a_sP_s(g_s^{q/p-2/p}\left(P_{1-s}|\nabla_{\mathbb{H}}f^p\right)^{2/p})\\
        &\leq 2\alpha(q-p)a_s \left(P_s g_s^{q/p}\right)^{1-2/q}\left(P_s (P_{1-s}|\nabla_{\mathbb{H}}f|^p)^{q/p}\right)^{2/q}\\
        &=2\alpha(q-p)a_s a_s^{-1}\left(P_s (P_{1-s}|\nabla_{\mathbb{H}}f|^p)^{q/p}\right)^{2/q}\\
        &=2\alpha(q-p)\left(P_s (P_{1-s}|\nabla_{\mathbb{H}}f|^p)^{q/p}\right)^{2/q}.
    \end{align*}
    Notice that since $1\leq p \leq q$,
    $$P_{1-s}|\nabla_\mathbb{H}f|^p \leq \left(P_{1-s}|\nabla_{\mathbb{H}}f|^q\right)^{p/q}\left(\int_{\mathbb{H}^N}p_N(\rho, \alpha(1-s))dV_{\mathbb{H}}\right)^{1-p/q},$$
    which allows us
    \begin{align*}
        \partial_s\phi_s &\leq 2\alpha(q-p)\left(\int_{\mathbb{H}^N}p_N(\rho, \alpha(1-s))dV_{\mathbb{H}}\right)^{2/p-2/q} \left(P_s(P_{1-s}|\nabla_{\mathbb{H}}f|^q)\right)^{2/q}\\
        &= 2\alpha(q-p)\left(\int_{\mathbb{H}^N}p_N(\rho, \alpha(1-s))dV_{\mathbb{H}}\right)^{2/p-2/q}(P_1|\nabla_{\mathbb{H}}f|^q)^{2/q}.
    \end{align*}
    Then, 
    \begin{align*}
        \partial_s\phi_s(0)&\leq 2\alpha(q-p)\left(\int_{\mathbb{H}^N}p_N(\rho, \alpha(1-s))dV_{\mathbb{H}}\right)^{2/p-2/q}\left(\int_{\mathbb{H}^N}|\nabla_{\mathbb{H}}f|^qd\gamma\right)^{2/q}\\
        &=2\alpha(q-p)\left(\int_{\mathbb{H}^N}|\nabla_{\mathbb{H}}f|^qd\gamma\right)^{2/q}.
    \end{align*}
    From \eqref{extBecknerineq}, 
    $$\left(\int_{\mathbb{H}^N}f^q d\gamma\right)^{2/q}-\left(\int_{\mathbb{H}^N}f^p d\gamma\right)^{2/p}=\int_0^1 \partial_s \phi_s(0)ds\leq 2\alpha(q-p)\left(\int_{\mathbb{H}^N}|\nabla_{\mathbb{H}}f|^qd\gamma\right)^{2/q},$$
    as desired.
\end{proof}
In particular, with $q=2$  and $p=1$, we get the following Poincar\'{e} inequality for the measure $\gamma$, i.e.
$$\int_{\mathbb{H}^N}f^2d\gamma -\left(\int_{\mathbb{H}^N}fd\gamma\right)^2 \leq 2\alpha\int_{\mathbb{H}^N}|\nabla_{\mathbb{H}}f|^2d\gamma.$$
Besides, again with $q=2$ and by letting $p$ approach to $2$, we get the logarithmic Sobolev inequality
\begin{corollary}
    $$\int_{\mathbb{H}^N}f^2\log(f^2)d\gamma \leq \log\left(\int_{\mathbb{H}^N}f^2d\gamma\right)\int_{\mathbb{H}^N}f^2d\gamma +4\alpha \int_{\mathbb{H}^N}|\nabla_{\mathbb{H}}f|^2d\gamma.$$
\end{corollary}
\begin{proof}
    We define
    $$F(p)=\left(\int_{\mathbb{H}^N}f^p d\gamma\right)^{2/p}.$$ Then,
    $$F'(p)=\left(-\dfrac{2}{p^2}\log\left(\int_{\mathbb{H}^N}f^pd\gamma\right)+\dfrac{2}{p}\dfrac{\int_{\mathbb{H}^N}f^p \log(|f|)d\gamma}{\int_{\mathbb{H}^N}f^pd\gamma}\right)\left(\int_{\mathbb{H}^N}f^pd\gamma\right)^{2/p},$$
    which implies that
    $$F'(2)=-\dfrac{1}{2}\log\left(\int_{\mathbb{H}^N}f^2d\gamma\right)\int_{\mathbb{H}^N}f^2d\gamma+\int_{\mathbb{H}^N}f^2\log(|f|)d\gamma.$$
    Hence, 
    $$2\alpha\int_{\mathbb{H}^N}|\nabla_{\mathbb{H}}f|^2d\gamma\geq F'(2)=-\dfrac{1}{2}\log\left(\int_{\mathbb{H}^N}f^2d\gamma\right)\int_{\mathbb{H}^N}f^2d\gamma+\int_{\mathbb{H}^N}f^2\log(|f|)d\gamma,$$
    as desired.
\end{proof}

\section{logarithmic Sobolev inequalities on Cartan-Hadamard model manifolds}

In this section, we will work on a more general setting, which is a model manifold. Let $\mathbb{M}^N$ be an $N$-dimensional Riemannian manifold ($N \geq 2$), called a 
\emph{model manifold}, whose metric can be written in the form
\begin{equation}\label{metric}
g = \mathrm{d}r \otimes \mathrm{d}r + \psi(r)^2 g_{\mathbb{S}^{N-1}},
\end{equation}
where $r$ denotes the geodesic distance from a fixed point $x_0 \in \mathbb{M}^N$, referred to as the 
\emph{pole} of the manifold. The vector $\mathrm{d}r$ represents the radial direction, and 
$g_{\mathbb{S}^{N-1}}$ is the standard metric on the unit sphere $\mathbb{S}^{N-1}$. 
The smooth function $\psi : [0,\infty) \to [0,\infty)$ completely determines the geometry of $\mathbb{M}^N$.

In this setting, every point $x \in \mathbb{M}^N \setminus \{x_0\}$ can be expressed in polar coordinates 
$(r,\theta) \in (0,\infty) \times \mathbb{S}^{N-1}$, where $r$ measures the distance to the pole $x_0$ 
and $\theta$ specifies the direction of the minimizing geodesic connecting $x$ to $x_0$. 
For general background on model manifolds, we refer to \cite[Section 3.10]{Gre24}. 

\medskip 

The function $\psi$ generating the metric is required to satisfy the following
regularity conditions in order to ensure smoothness of the Riemannian structure:
\begin{equation}\label{psi}
\psi \in C^\infty([0,\infty)), 
\qquad \psi(r) > 0 \ \text{for } r > 0, 
\qquad \psi'(0) = 1, 
\qquad \psi^{(2k)}(0) = 0 \ \ \forall\, k \in \mathbb{N} \cup \{0\}.
\end{equation}
These assumptions are known to be necessary and sufficient for the smoothness of 
the manifold $\mathbb{M}^N$. In many analytical arguments, only the conditions 
$\psi(0) = 0$ and $\psi'(0) = 1$ are essential; however, without the full set of
conditions in \eqref{psi}, the metric may fail to be regular at the pole.

By prescribing the function $\psi$ on the entire half-line $[0,\infty)$, we implicitly 
assume that the associated model manifold is complete and noncompact. This is precisely 
the geometric setting relevant to our analysis. By the Cartan--Hadamard theorem, any such 
manifold is homeomorphic to Euclidean space $\mathbb{R}^N$; moreover, the exponential 
map at any point defines a global diffeomorphism.

Classical examples of noncompact Riemannian model manifolds include Euclidean space 
$\mathbb{R}^N$ and hyperbolic space $\mathbb{H}^N$, corresponding respectively to the 
choices $\psi(r) = r$ and $\psi(r) = \sinh r$.

\medskip

{\bf Throughout this section we further assume that $\psi$ is a \emph{convex} function.} Under this 
assumption, the resulting manifold $\mathbb{M}^N$ becomes a Cartan--Hadamard manifold, 
that is, a complete and simply connected Riemannian manifold with nonpositive sectional 
curvature. For further background we refer the reader to \cite{GW}. Convexity of $\psi$ 
is equivalent to 
\begin{equation}\label{convex-assump}
    \psi''(r) \ge 0 \qquad \text{for all } r>0,
\end{equation}
and in particular implies the bound
\begin{equation}\label{ins-assump}
    \psi'(r) \ge 1 \qquad \text{for all } r>0.
\end{equation}
For any $\delta>0$, let
\[
B_\delta(x_0):=\{x\in \mathbb{M}^N:\ \mathrm{dist}(x,x_0)<\delta\}
\]
denote the geodesic ball in $\mathbb{M}^N$ of radius $\delta$ centered at a fixed pole $x_0$.
If $u\in L^{1}(\mathbb{M}^N)$, then in geodesic polar coordinates $(r,\theta)$ about $x_0$ one has
\[
\int_{\mathbb{M}^N} u(x)\, d\mathrm{vol}_g
=
\int_{\mathbb{S}^{N-1}} \int_{0}^{\infty}
u(r,\theta)\, (\psi(r))^{N-1}\, dr\, d\theta,
\]
where $d\theta$ denotes the standard surface measure on the unit sphere $\mathbb{S}^{N-1}$.

\medskip

\subsection{Schwarz rearrangements and P\'olya-Szeg\"o inequality.} We start by reviewing the concept of rearrangements on a manifold $\mathbb{M}^N$, together with the fundamental measure-theoretic principles that form the basis of symmetrization techniques.
 Let $u: \mathbb{M}^N \rightarrow \mathbb{R}$ be a function such that the volume
$$\text{Vol}_g(\{x \in \mathbb{M}^N:|u(x)|>t\}):=V(\{x \in \mathbb{M}^N:|u(x)|>t\})< \infty\;\text{for all $t>0$}.$$
We define the \textbf{distribution function $\mu_u$} of $u$ as follows:
$$\mu_u(t)=V(\{x \in \mathbb{M}^N:|u(x)|>t\}),$$
which is non-increasing and right-continuous. Then, the \textbf{decreasing rearrangement function $u^*$} of $u$ is defined by
$$u^*(t)=\sup \{s>0: \mu_u(s)>t\},$$
which is non-increasing. Besides, we also define some symmetric decreasing rearrangement functions as follows:
$$u_g^\sharp(x):=u^*(V(B_g(o, \rho(x)))),\; \text{for all $x \in \mathbb{M}^N$},$$
and
$$u_e^\sharp(x):=u^*(\sigma_N |x|^N),\; \text{for all $x \in \mathbb{R}^N$},$$
where $\sigma_N$ denotes the volume of unit ball in $\mathbb{R}^N$.
\begin{lemma}[Comparison of gradient norms of rearrangements]\label{lemC}
Let $u\in W^{1,p}(\mathbb{M}^N)$ and denote by $u_g^\sharp$ and $u_e^\sharp$ the symmetric decreasing rearrangements of $u$ on $\mathbb{M}^N$ and $\mathbb{R}^N$, respectively.  
We also set $v:=u^*$, the decreasing rearrangement of $u$. Then the following identity holds:
\begin{align*}
\int_{\mathbb{M}^N} |\nabla_g u_g^\sharp|_g^p \, dV
=
\int_{\mathbb{R}^N} |\nabla u_e^\sharp|^p \, dx
+
(N\sigma_N)^p
\int_0^\infty |v'(s)|^p \, k_{N,p}\!\left(\frac{s}{\sigma_N}\right) ds,
\end{align*}
where
\[
k_{N,p}(s)
:=
\bigl(\psi(\Phi^{-1}(s))\bigr)^{p(N-1)}
-
s^{\frac{p(N-1)}{N}},
\qquad
\Phi(t):=N\int_0^t (\psi(\tau))^{N-1}\, d\tau .
\]
\end{lemma}

\begin{proof}
We first recall that the Euclidean rearrangement satisfies
\begin{equation}\label{eq:Eucl-grad}
\int_{\mathbb{R}^N} |\nabla u_e^\sharp|^p \, dx
=
(N\sigma_N)^p
\int_0^\infty |v'(s)|^p
\left(\frac{s}{\sigma_N}\right)^{\frac{(N-1)p}{N}} ds .
\end{equation}

On the manifold $\mathbb{M}^N$, the volume of a geodesic ball of radius $\rho(x)$ is given by
\[
V(B_g(o,\rho(x)))
=
\int_{\mathbb{S}^{N-1}}\int_0^{\rho(x)} (\psi(t))^{N-1}\, dt\, d\sigma
=
\sigma_N \Phi(\rho(x)),
\]
where $\Phi(t)=N\int_0^t (\psi(t))^{N-1}dt$. Consequently,
\[
\nabla_g V(B_g(o,\rho(x)))
=
N\sigma_N (\psi(\rho(x)))^{N-1}\nabla_g \rho(x).
\]

Since $u_g^\sharp(x)=v(V(B_g(o,\rho(x))))$, we compute
\begin{align*}
\int_{\mathbb{M}^N} |\nabla_g u_g^\sharp|_g^p \, dV
&=
\int_{\mathbb{M}^N}
\bigl| v'(V(B_g(o,\rho(x)))) \bigr|^p
\bigl| \nabla_g V(B_g(o,\rho(x))) \bigr|_g^p \, dV \\
&=
N\sigma_N
\int_0^\infty
|v'(V(B_g(o,t)))|^p
\bigl(N\sigma_N \psi^{N-1}(t)\bigr)^p
\psi^{N-1}(t)\, dt .
\end{align*}

Performing the change of variables
\[
s = V(B_g(o,t)) = \sigma_N \Phi(t),
\qquad
ds = N\sigma_N \psi^{N-1}(t)\, dt,
\]
we obtain
\[
\int_{\mathbb{M}^N} |\nabla_g u_g^\sharp|_g^p \, dV
=
(N\sigma_N)^p
\int_0^\infty
|v'(s)|^p
\bigl(\psi(\Phi^{-1}(s/\sigma_N))\bigr)^{p(N-1)} ds .
\]

Comparing this with \eqref{eq:Eucl-grad} yields the claimed decomposition, completing the proof.
\end{proof}

We begin by examining the function
\[
k_{N,p}(s)
:=
\bigl(\psi(\Phi^{-1}(s))\bigr)^{p(N-1)}
-
s^{\frac{p(N-1)}{N}},
\qquad s\ge 0 .
\]
To understand the behavior of $k_{N,p}$, it is natural to study the asymptotics of the following quotient near the origin and at infinity. In particular, we are led to consider the following quantity 
\[
\frac{\psi^{p(N-1)}(t)-\Phi^{\frac{p(N-1)}{N}}(t)}{\Phi^{p}(t)} .
\]

\begin{lemma}[Asymptotic behavior]\label{lem:CNP-asymptotics}
Let $N\ge 2$ and $p\ge \frac{N}{N-1}$. Assume that $\psi^{\prime \prime \prime}(0)>0$
and that the limit
\[
\lim_{t\to\infty}\frac{\psi'(t)}{\psi(t)}=:C_1
\]
exists. Then the following hold.
\begin{enumerate}
\item[\emph{(i)}] (\emph{Behavior near the origin})
\[
\lim_{t\to 0}
\frac{\psi^{p(N-1)}(t)-\Phi^{\frac{p(N-1)}{N}}(t)}{\Phi^{p}(t)}
=
\begin{cases}
0, & \text{if } p<2,\\[4pt]
\dfrac{6(N-1)}{N+2}\,a_3, & \text{if } p=2,\\[6pt]
+\infty, & \text{if } p>2.
\end{cases}
\]
Here $a_3= \psi^{\prime \prime \prime}(0).$
\item[\emph{(ii)}] (\emph{Behavior at infinity})
\[
\lim_{t\to\infty}
\frac{\psi^{p(N-1)}(t)-\Phi^{\frac{p(N-1)}{N}}(t)}{\Phi^{p}(t)}
=
\left(\frac{N-1}{N}\right)^p C_1^p .
\]
\end{enumerate}
\end{lemma}
\medskip

\begin{proof}
{\bf Case $p=2$:}
First, we compute
$$\lim_{t \rightarrow 0} \dfrac{\psi^{2(N-1)}(t)-\Phi^{\frac{2(N-1)}{N}}(t)}{\Phi^2(t)}.$$  
Applying L'Hospital rule twice, it is straightforward to get
$$\lim_{t \rightarrow 0} \dfrac{\psi^{2(N-1)}(t)-\Phi^{\frac{2(N-1)}{N}}(t)}{\Phi^2(t)}=\dfrac{(N-1)(N-2)}{N^2}\lim_{t\rightarrow 0} \left[\left(\dfrac{\psi'(t)}{\psi(t)}\right)^2+\dfrac{1}{N-2}\dfrac{\psi''(t)}{\psi(t)}-\Phi^{-2/N}(t)\right],$$
In order to evaluate the limit
$$\lim_{t\rightarrow 0} \left[\left(\dfrac{\psi'(t)}{\psi(t)}\right)^2+\dfrac{1}{N-2}\dfrac{\psi''(t)}{\psi(t)}-\Phi^{-2/N}(t)\right],$$
 we can represent $\psi$ using Taylor's series as
$$\psi(t)=t+a_3t^3+o(t^3)\;\text{as $t \rightarrow 0$},$$
where $a_3>0$. By computing directly, we have
\begin{align*}
    \dfrac{\psi'(t)}{\psi(t)}&=\dfrac{1}{t}(1+3a_3t^2+o(t^2))(1+a_3t^2+o(t^2))^{-1}
    =\dfrac{1}{t}(1+3a_3t^2+o(t^2))(1-a_3t^2+o(t^2))\\
    &=\dfrac{1}{t}(1+2a_3t^2+o(t^2)),
\end{align*}
which allows us to write 
$$\left(\dfrac{\psi'(t)}{\psi(t)}\right)^2=\dfrac{1}{t^2}+4a_3+o(1).$$
Furthermore,
\begin{align*}
    \dfrac{\psi''(t)}{\psi(t)}&=\dfrac{1}{t}\left(6a_3t+o(t)\right)(1+a_3t^2+o(t^2))^{-1}
    =\dfrac{1}{t}\left(6a_3t+o(t)\right)(1-a_3t^2+o(t^2))\\
    &=\dfrac{1}{t}(6a_3t+o(t))=6a_3+o(1),
\end{align*}
and
\begin{align*}
    \Phi(t)=N\int_0^t \psi^{N-1}(s)ds&=N\int_0^t (s+a_3s^3+o(s^3))^{N-1}ds
    =N \int_0^t s^{N-1}(1+a_3s^2+o(s^2))^{N-1}ds\\
    &=N\int_0^t s^{N-1}(1+(N-1)a_3s^2+o(s^2))ds\\
    &=t^N+\dfrac{N(N-1)a_3}{N+2}t^{N+2}+o(t^{N+2}),
\end{align*}
which implies that
$$(\Phi(t))^{-2/N}=\left(t^N+\dfrac{N(N-1)a_3}{N+2}t^{N+2}+o(t^{N+2})\right)^{-2/N}=\dfrac{1}{t^2}\left(1-\dfrac{2(N-1)}{N+2}a_3t^2+o(t^2)\right).$$
Hence, 
\begin{align*}
    &\left(\dfrac{\psi'(t)}{\psi(t)}\right)^2+\dfrac{1}{N-2}\dfrac{\psi''(t)}{\psi(t)}-\Phi^{-2/N}(t)\\
    &=\dfrac{1}{t^2}+4a_3+o(1)+\dfrac{1}{N-2}(6a_3+o(1))-\dfrac{1}{t^2}\left(1-\dfrac{2(N-1)}{N+2}a_3t^2+o(t^2)\right),
\end{align*}
which tends to
$\left(4+\dfrac{6}{N-2}+\dfrac{2(N-1)}{N+2}\right)a_3=\dfrac{6N^2}{N^2-4}a_3$ as $t \rightarrow 0$.
Thus,
$$\lim_{t \rightarrow 0} \dfrac{\psi^{2(N-1)}(t)-\Phi^{\frac{2(N-1)}{N}}(t)}{\Phi^2(t)}=\dfrac{(N-1)(N-2)}{N^2}\dfrac{6N^2}{N^2-4}a_3=\dfrac{6(N-1)}{N+2}a_3.$$

\medskip 

Now we compute the limit at infinity. We have
\begin{align*}
\lim_{t\to\infty}
\frac{\psi^{2(N-1)}(t)-\Phi^{\frac{2(N-1)}{N}}(t)}{\Phi^{2}(t)}
&=
\lim_{t\to\infty}
\left(\frac{\psi^{N-1}(t)}{\Phi(t)}\right)^{2} \\
&=
\left(
\lim_{t\to\infty}
\frac{\psi^{N-1}(t)}{\Phi(t)}
\right)^{2}
=
\left(\frac{N-1}{N}\right)^{2}
\left(
\lim_{t\to\infty}
\frac{\psi'(t)}{\psi(t)}
\right)^{2} \\
&=
\left(\frac{N-1}{N}\right)^{2} C_1^{2}.
\end{align*}
\medskip 

{\bf For general $p \geq \frac{N}{N-1}$:} We compute
$$\lim_{t \rightarrow 0} \dfrac{\psi^{p(N-1)}(t)-\Phi^{\frac{p(N-1)}{N}}(t)}{\Phi^p(t)}.$$
Applying L'Hospital rule twice, it is straightforward to get
\begin{align*}
    &\lim_{t \rightarrow 0} \dfrac{\psi^{p(N-1)}(t)-\Phi^{\frac{p(N-1)}{N}}(t)}{\Phi^p(t)}\\
    &=\dfrac{(N-1)((p-1)N-p)}{(p-1)N^2}\\
    &\times\lim_{t\rightarrow 0} \left[\dfrac{(\psi(t))^{(p-2)N-p}(\psi'(t))^2}{(\Phi(t))^{p-2}}+\dfrac{1}{(p-1)N-p}\dfrac{(\psi(t))^{(p-2)N+1-p}\psi''(t)}{(\Phi(t))^{p-2}}-\Phi^{-p/N}(t)\right].
\end{align*}
As above, we assume that $\psi(t)=t+a_3t^3+o(t^3)$ as $t \rightarrow 0$, we have
$$\psi(t)^{(p-2)N-p}=t^{(p-2)N-p}\left(1+((p-2)N-p)a_3t^2+o(t^2)\right),$$
$$\psi(t)^{(p-2)N-p+1}=t^{(p-2)N-p+1}\left(1+((p-2)N-p+1)a_3t^2+o(t^2)\right),$$
\and
$$(\psi'(t))^2=1+6a_3t^2+o(t^2).$$
Since $\Phi(t)=t^N+\dfrac{N(N-1)}{N+2}a_3 t^{N+2}+o(t^{N+2})$, we have
$$\Phi(t)^{2-p}=t^{N(2-p)}\left(1+(2-p)\dfrac{N(N-1)}{N+2}a_3t^2+o(t^2)\right),$$
and
$$\Phi(t)^{-p/N}=t^{-p}\left(1-\dfrac{p(N-1)}{N+2}a_3t^2+o(t^2)\right).$$
Then,
\begin{align*}
    &\dfrac{(\psi(t))^{(p-2)N-p}(\psi'(t))^2}{(\Phi(t))^{p-2}}+\dfrac{1}{(p-1)N-p}\dfrac{(\psi(t))^{(p-2)N+1-p}\psi''(t)}{(\Phi(t))^{p-2}}-\Phi^{-p/N}(t)\\
    &=\left((p-2)\dfrac{3N}{N+2}+6-p+\dfrac{6}{(p-1)N-p}+\dfrac{p(N-1)}{N+2}\right)a_3t^{2-p}+o(t^{2-p}),
\end{align*}
which tends to $0$ if $p<2$ and to $\infty$ as $p>2$.

Moreover, 
\begin{align*}
    \lim_{t \rightarrow \infty} \dfrac{\psi^{p(N-1)}(t)-\Phi^{\frac{p(N-1)}{N}}(t)}{\Phi^p(t)}=\lim_{t \rightarrow \infty} \left(\dfrac{\psi^{N-1}(t)}{\Phi(t)}\right)^p= \left(\lim_{t \rightarrow \infty}\dfrac{\psi^{N-1}(t)}{\Phi(t)}\right)^p&=\left(\dfrac{N-1}{N}\right)^p \left(\lim_{t\rightarrow \infty} \dfrac{\psi'(t)}{\psi(t)} \right)^p\\
    &=\left(\dfrac{N-1}{N}\right)^pC_1^p.
\end{align*}
\end{proof}


\medskip
As a direct consequence, we have the following corollary.
\begin{corollary}[Lower bound for the correction term]\label{lemlower}
Let $N\geq 2$ and $p\geq \frac{N}{N-1}$. 
 Assume that
\begin{equation}\label{cdt 8.4}
    C_1= \lim_{t\rightarrow \infty} \dfrac{\psi'(t)}{\psi(t)}>0,
\end{equation}
and
\begin{equation}\label{cdt 8.5}
    k_{N,p}(\Phi(t))=(\psi(t))^{p(N-1)}-(\Phi(t))^{\frac{p(N-1)}{N}} >0\;\text{for all $t>0$}.
\end{equation}
Then there exists a constant $C(N,p)>0$ such that
\[
k_{N,p}(s)\ge C(N,p)\, s^{p}
\qquad \text{for all } s\ge 0 ,
\]
which is equivalent to
\[
\psi^{p(N-1)}(t)-\Phi^{\frac{p(N-1)}{N}}(t)
\;\ge\;
C(N,p)\,\Phi^{p}(t)
\qquad \text{for all } t\ge 0 ,
\]
where $\Phi(t)=N\int_0^t (\psi(\tau))^{N-1}\, d\tau$.
Moreover, the optimal constant is given by
\[
C(N,p)
=
\inf_{t\ge 0}
\frac{\psi^{p(N-1)}(t)-\Phi^{\frac{p(N-1)}{N}}(t)}{\Phi^{p}(t)} .
\]
\end{corollary}

\begin{remark}\label{rem:failure-and-sufficient}
{\rm
In general, condition~\eqref{cdt 8.5} does \emph{not} hold without additional assumptions on the function $\psi$.
Indeed, suppose that
\[
\psi(t)=t+a_3 t^3+o(t^3)
\qquad \text{as } t\to 0.
\]
Here $a_3 = \psi^{\prime \prime \prime}(0).$
A direct expansion yields
\begin{align*}
(\psi(t))^{p(N-1)}
&=
t^{p(N-1)}\bigl(1+a_3 t^2+o(t^2)\bigr)^{p(N-1)} \\
&=
t^{p(N-1)}
+
p(N-1)a_3\, t^{p(N-1)+2}
+
o\!\left(t^{p(N-1)+2}\right),
\end{align*}
while
\begin{align*}
(\Phi(t))^{\frac{p(N-1)}{N}}
&=
t^{p(N-1)}
\left(1+\frac{N(N-1)}{N+2}a_3 t^2+o(t^2)\right)^{\frac{p(N-1)}{N}} \\
&=
t^{p(N-1)}
+
\frac{p(N-1)^2}{N+2}a_3\,
t^{p(N-1)+2}
+
o\!\left(t^{p(N-1)+2}\right).
\end{align*}
Consequently,
\[
k_{N,p}(\Phi(t))
=
\frac{3p(N-1)}{N+2}\, a_3\,
t^{p(N-1)+2}
+
o\!\left(t^{p(N-1)+2}\right),
\]
which shows that $k_{N,p}(\Phi(t))$ need not be nonnegative for all $t>0$.
In particular, choosing for instance $\psi(t)=t+t^3-t^5$ provides an explicit counterexample.
}
\end{remark}

\begin{remark}
{\rm
On the other hand, condition~\eqref{cdt 8.5} is satisfied under suitable monotonicity assumptions on $\psi$.
If $\psi'(t)>1$ for all $t>0$, then
\[
\Phi(t)
=
N\int_0^t (\psi(\tau))^{N-1}\, d\tau
<
N\int_0^t (\psi(\tau))^{N-1}\psi'(\tau)\, d\tau
=
(\psi(t))^{N},
\]
which directly implies~\eqref{cdt 8.5}.
A sufficient condition ensuring $\psi'(t)>1$ in a neighborhood of the origin is, for instance, $\psi''(t)>0$ near $t=0$, since then $\psi'(t)>\psi'(0)=1$.
}
\end{remark}
\medskip
\begin{lemma}[Sufficient condition for attainment at \(t=0\)]\label{lem:attainment-zero}
Let \(N\ge 3\) and $\psi$ satisfies \eqref{psi}, and \eqref{cdt 8.4}.
Define
\[
K(t):=\left(\frac{\psi'(t)}{\psi(t)}\right)^{2}
+\frac{1}{N-2}\frac{\psi''(t)}{\psi(t)}
-\frac{6N^{2}}{N^{2}-4}a_{3}.
\]
If
\begin{equation}\label{eq cdt 3}
\psi^{N-1}(t)+\frac{1}{2}\frac{K'(t)}{K^{\frac{N+2}{2}}(t)}\ge 0
\qquad \text{for all } t\ge 0,
\end{equation}
then the infimum in the definition of \(C(N,2)\) is attained at \(t=0\).
\end{lemma}

\begin{proof}
We aim to find a sufficient condition on $\psi$ such that the infimum in the
definition of $C(N,2)$ is attained at $t=0$, namely
\[
\psi^{2(N-1)}(t)-\Phi^{\frac{2(N-1)}{N}}(t)
\geq \frac{6(N-1)}{N+2}a_3\Phi^2(t),
\qquad t\ge 0,\; N\ge 3.
\]
Define
\[
F(t):=\psi^{2(N-1)}(t)-\Phi^{\frac{2(N-1)}{N}}(t)
-\frac{6(N-1)}{N+2}a_3\Phi^2(t).
\]
Since $F(0)=0$, it is enough to show that $F'(t)\ge 0$ for all $t\ge 0$.
A direct computation gives
\[
F'(t)
=2(N-1)\psi^{N-1}(t)
\left(
\psi^{N-2}(t)\psi'(t)
-\Phi^{\frac{N-2}{N}}(t)
-\frac{6N}{N+2}a_3\Phi(t)
\right)
=:2(N-1)\psi^{N-1}(t)G(t).
\]
Hence, $F'(t)\ge 0$ follows if $G(t)\ge 0$. Since $G(0)=0$, it suffices to
ensure that $G'(t)\ge 0$ for all $t\ge 0$.

Differentiating $G$, we obtain
\begin{align*}
G'(t)
&=(N-2)\psi^{N-3}(t)(\psi'(t))^2
+\psi^{N-2}(t)\psi''(t) \\
&\quad -(N-2)\Phi^{-2/N}(t)\psi^{N-1}(t)
-\frac{6N^2}{N+2}a_3\psi^{N-1}(t).
\end{align*}
Factoring out $\psi^{N-1}(t)$, we see that $G'(t)\ge 0$ provided
\[
\Phi(t)\ge
\left(
\left(\frac{\psi'(t)}{\psi(t)}\right)^2
+\frac{1}{N-2}\frac{\psi''(t)}{\psi(t)}
-\frac{6N^2}{N^2-4}a_3
\right)^{-N/2}.
\]
Set
\[
K(t):=
\left(\frac{\psi'(t)}{\psi(t)}\right)^2
+\frac{1}{N-2}\frac{\psi''(t)}{\psi(t)}
-\frac{6N^2}{N^2-4}a_3,
\qquad
H(t):=\Phi(t)-K(t)^{-N/2}.
\]
From the expansion of $\psi$ near $t=0$, we obtain
\[
K^{-1}(t)
=t^2\left(1+2\frac{N-1}{N+2}a_3t^2+o(t^2)\right)\to 0,
\qquad t\to 0,
\]
and hence $H(t)\to 0$ as $t\to 0$. Therefore, $H(t)\ge 0$ for all $t\ge 0$
follows if $H'(t)\ge 0$.

A direct computation shows that
\[
H'(t)\ge 0
\quad\Longleftrightarrow\quad
N\psi^{N-1}(t)
+\frac{N}{2}\frac{K'(t)}{K(t)^{(N+2)/2}}
\ge 0,
\]
which is equivalent to
\begin{equation}\label{eq cdt 3}
\psi^{N-1}(t)
+\frac{1}{2}\frac{K'(t)}{K(t)^{\frac{N+2}{2}}}
\ge 0.
\end{equation}
Thus, condition \eqref{eq cdt 3} is sufficient to guarantee $F(t)\ge 0$ for
all $t\ge 0$, and consequently the infimum in the definition of $C(N,2)$ is
attained at $t=0$.
\end{proof}

\medskip

\subsection{Centered isoperimetric inequality}
Our approach to establishing the logarithmic Sobolev inequality relies on an improved version of the P\'olya--Szeg\"o  inequality. Although a highly nontrivial problem is to determine when the P\'olya–Szeg\"o inequality remains valid in the setting of model manifolds. 
Let $\dot H^{1}(\mathbb{M}^n)$ denotes the space of functions
$v \in L^2_{\mathrm{loc}}(\mathbb{M}^N)$ such that $\nabla v \in L^2(\mathbb{M}^N)$.

\begin{definition}[P\'olya--Szeg\"o inequality]\label{def:polya-szego}
We say that a noncompact model manifold $M^n$ supports the \emph{P\'olya--Szeg\"o inequality} if, for every function
$v \in \dot H^{1}(\mathbb{M}^N)$, its symmetric decreasing rearrangement $v^\star$ also belongs to
$\dot H^{1}(\mathbb{M}^N)$ and satisfies
\[
\int_{\mathbb{M}^N} |\nabla v^\star|^2 \, dV
\;\le\;
\int_{\mathbb{M}^N} |\nabla v|^2 \, dV .
\]
\end{definition}
Muratori and Volzone~\cite[Theorem~3.9]{MV25} showed that this inequality is not always true in model manifolds: they constructed explicit families of model manifolds for which the P\'olya–Szeg\"o principle breaks down. At the same time, their analysis clarifies the geometric condition  behind this phenomenon. In particular, they derived explicit geometric hypotheses under which the inequality is restored. A key outcome of their work is the identification of a centered isoperimetric inequality as a sufficient condition. In particular, they proved that for any model manifold, \eqref{def:polya-szego} holds whenever $\mathbb{M}^N$ supports a centered isoperimetric inequality, namely,
\begin{align}\label{isoper}
\text{Per}(B_r(x_0)) \leq \text{Per}(\Omega) \quad \forall \,  \Omega \in \mathcal{B}_b(\mathbb{M}^N),
\end{align}
where $\text{Vol}_g(B_r(x_0)) = V(r) = \text{Vol}(\Omega)$ and $\mathcal{B}_b(\mathbb{M}^N)$ denotes the family of all bounded Borel sets in $\mathbb{M}^N$ (see \cite[Proposition~3.5]{MV25}). 
\medskip

Now we shall establish the following improved version of  P\'olya–Szeg\"o in $\mathbb{M}^N.$ In what follows, we will consistently assume that $a_3:= \psi^{\prime\prime\prime}(0) > 0$.

\begin{theorem}
   Let $\mathbb{M}^N$ be an $N$-dimensional model manifold endowed with the
rotationally symmetric metric where the warping function $\psi$ satisfies \eqref{psi}, \eqref{eq cdt 3},
and $a_3=\psi'''(0)>0$. Assume that $\mathbb{M}^N$ supports a \emph{centered
isoperimetric inequality} \eqref{isoper}. Then,
    $$\int_{\mathbb{M}^N}|\nabla_g u|_g^2dV-\dfrac{3N^2(N-1)}{2(N+2)}a_3\int_{\mathbb{M}^N}|u|^2_gdV \geq \int_{\mathbb{R}^N}|\nabla u_e^\sharp|^2dx.$$
\end{theorem}

\begin{proof}
Using Lemma~\ref{lemC} and Corollary~\ref{lemlower} we have 
\begin{align*}
    \int_{\mathbb{M}^N}|\nabla_g u_g^\sharp (x)|_g^2dV&=\int_{\mathbb{R}^N}|\nabla u_e^\sharp|^2dx+(N\sigma_N)^2 \int_0^\infty |v'(s)|^2 k_{N,2}(s/\sigma_N)ds\\
    & \geq \int_{\mathbb{R}^N}|\nabla u_e^\sharp|^2dx+\dfrac{6N^2(N-1)}{N+2}a_3 \int_0^\infty |v'(s)|^2 s^2ds\\
    & = \int_{\mathbb{R}^N}|\nabla u_e^\sharp|^2dx+\dfrac{6N^2(N-1)}{N+2}a_3 \left(\int_0^\infty |w'(s)|^2 sds+\dfrac{1}{4}\int_0^\infty (v(s))^2ds\right),
\end{align*}
where $w(s):=v(s)s^{1/2}$. Here, we used $w(0)=w(\infty)=0$. Thus,
$$\int_{\mathbb{M}^N}|\nabla_g u_g^\sharp (x)|_g^2dV-\dfrac{3N^2(N-1)}{2(N+2)}a_3\int_{\mathbb{M}^N}|u_g^\sharp|^2_gdV \geq \int_{\mathbb{R}^N}|\nabla u_e^\sharp|^2dx.$$

\end{proof}

As a direct consequence, we obtain
\begin{corollary}
    Under the same assumptions as the above theorem, for any $u \in W^{1,2}(\mathbb{M}^N)$ such that $\|u\|_{L^2(\mathbb{M}^N)}=1$, we have
    $$\int_{\mathbb{M}^N} |u|^2\log|u|dV \leq \dfrac{N}{4}\log\left[\mathcal{L}_2\left(\int_{\mathbb{M}^N}|\nabla_gu|_g^2dV-\dfrac{3N^2(N-1)}{2(N+2)}a_3\int_{\mathbb{M}^N}|u|_g^2dV\right)\right],$$
    where the constant $\mathcal{L}_2$ is from \cite{PD03}.
\end{corollary}

As an immediate consequence of Corollary~\ref{lemlower}, we obtain the following result.
\begin{theorem}
Let $\mathbb{M}^N$ be an $N$-dimensional model manifold endowed with the
rotationally symmetric metric where the warping function $\psi$ satisfies \eqref{psi}, \eqref{cdt 8.4}, \eqref{cdt 8.5} and $a_3=\psi'''(0)>0$. Assume that $\mathbb{M}^N$ supports a \emph{centered
isoperimetric inequality} \eqref{isoper}.

Then, for every $N \ge 2$ and $p>2$, the following inequality holds:
\[
\int_{\mathbb{M}^N} |\nabla_g u|_g^p \, dV
- C(N,p)\frac{N^p}{p^p}
\int_{\mathbb{M}^N} |u|^p \, dV
\ge
\int_{\mathbb{R}^N} |\nabla u_e^\sharp|^p \, dx .
\]
\end{theorem}

To sum up, we  prove the following version of  $p$-log-Sobolev inequality:
    \begin{theorem}
       Let $\mathbb{M}^N$ be an $N$-dimensional model manifold endowed with the
rotationally symmetric metric where the warping function $\psi$ satisfies \eqref{psi}, \eqref{cdt 8.4}, \eqref{cdt 8.5} and $a_3=\psi'''(0)>0$. Assume that $\mathbb{M}^N$ supports a \emph{centered
isoperimetric inequality} \eqref{isoper}. Let $2<p<N$. Then, for all $u \in W^{1,p}(\mathbb{M}^N)$ with $\|u\|_{L^p(\mathbb{M}^N)}=1$, we have
    $$\int_{\mathbb{M}^N} |u|^p\log|u|dV \leq \dfrac{N}{p^2}\log\left[\mathcal{L}_p\left(\int_{\mathbb{M}^N}|\nabla_g u|_g^pdV-C(N,p)\dfrac{N^p}{p^p}\int_{\mathbb{M}^N}|u|^p_gdV\right)\right],$$
    where $\mathcal{L}_p$ is from \cite{PD03}.
\end{theorem}
\begin{proof}
    By the same argument as in the proof of Theorem \ref{T1}, we obtain that for all $p>2$,
$$\int_{\mathbb{M}^N}|u|^p\log |u| dV=\int_{\mathbb{R}^N}(u_e^\sharp)^p\log (u_e^\sharp) dx.$$
From Theorem $1.1$ in \cite{PD03}, it follows us
\begin{align*}
        \int_{\mathbb{M}^N} |u|^p\log|u|dV&=\int_{\mathbb{R}^N}(u_e^\sharp)^p\log (u_e^\sharp) dx\\
        &\leq \dfrac{N}{p^2} \log \left[\mathcal{L}_p \int_{\mathbb{R}^N}|\nabla (u_e^\sharp)|^pdx\right]\\
        &\leq \dfrac{N}{p^2}\log\left[\mathcal{L}_p\left(\int_{\mathbb{M}^N}|\nabla_g u|_g^pdV-C(N,p)\dfrac{N^p}{p^p}\int_{\mathbb{M}^N}|u|^p_gdV\right)\right],\nonumber
    \end{align*}
    as desired.
\end{proof}
\medskip 

\section*{Acknowledgments} 
 A. Do and G. Lu were partially supported by grants from the
Simons Foundation. D. Ganguly was partially supported by the
SERB MATRICS (MTR/2023/000331). N. Lam was partially supported by an NSERC
Discovery Grant.

\end{document}